\documentclass[11pt, letterpaper]{amsart}
\usepackage[left=1in,right=1in,bottom=1in,top=1in]{geometry}
\usepackage{amsmath,amsthm,amssymb}
\usepackage{mathrsfs}
\usepackage[all,cmtip]{xy}
\usepackage{cite}
\usepackage{hyperref}
\usepackage{enumerate}
\usepackage{graphicx}
\usepackage{mathtools,caption} %oxford packages

\newcommand{\on}[1]{\operatorname{#1}}
\newcommand{\mathfont}{\mathbf}
\newcommand{\ZZ}{\mathfont Z}

\newcommand{\NN} {\mathfont N}
\newcommand{\FF}{\mathfont F}

\newcommand{\PP}{\mathfont{P}}
\newcommand{\GG}{\mathfont{G}}

\DeclareFontFamily{OT1}{rsfs}{}
\DeclareFontShape{OT1}{rsfs}{n}{it}{<-> rsfs10}{}
\DeclareMathAlphabet{\mathscr}{OT1}{rsfs}{n}{it}

\newcommand{\Gscr}{\mathscr{G}}

\renewcommand{\Im}{\on{Im}}
\newcommand{\into}{\hookrightarrow}

\newcommand{\ord}{\on{ord}}

\theoremstyle{plain}
\newtheorem{lem}{Lemma}
\newtheorem{thm}[lem]{Theorem}
\newtheorem{prop}[lem]{Proposition}
\newtheorem{cor}[lem]{Corollary}

\newtheorem{fact}[lem]{Fact}

\theoremstyle{definition}
\newtheorem{defn}[lem]{Definition}
\newtheorem{example}[lem]{Example}
\newtheorem{remark}[lem]{Remark}
\newtheorem{notation}[lem]{Notation}

\numberwithin{equation}{section}
\numberwithin{lem}{section}

\newcommand{\Cartier}{\mathscr{C}}
\newcommand{\Ga}{\mathfont{G}_a}

\DeclarePairedDelimiter\abs{\lvert}{\rvert}
\DeclarePairedDelimiter\norm{\lVert}{\rVert}
\DeclarePairedDelimiter\ceil{\lceil}{\rceil}
\DeclarePairedDelimiter\floor{\lfloor}{\rfloor}

\makeatletter
\let\oldabs\abs
\def\abs{\@ifstar{\oldabs}{\oldabs*}}

\let\oldnorm\norm
\def\norm{\@ifstar{\oldnorm}{\oldnorm*}}

\let\oldfloor\floor
\def\floor{\@ifstar{\oldfloor}{\oldfloor*}}

\let\oldceil\ceil
\def\ceil{\@ifstar{\oldceil}{\oldceil*}}

\newcommand{\sC}[1]{s\left(\mathcal{C}\left(#1\right)\right)}

\newcommand{\integers}{\mathbb{Z}}

\newcommand{\D}{d} %changed to a normal d to be compatible with rest of paper, and because it was used inconsistently
\DeclareMathOperator{\Ima}{Im}

\newcommand{\boldheader}[1]{\textbf{#1}.}

\makeatother

\title{Realizing Artin-Schreier Covers with Minimal $a$-numbers in Positive Characteristic}
\author[ABBMFMMNPXY]{Fiona Abney-McPeek \and Hugo Berg \and Jeremy Booher \and Sun Mee Choi \and Viktor Fukala \and Miroslav Marinov \and
Theo M\"{u}ller \and Pawe\l{} Narkiewicz \and Rachel Pries \and Nancy Xu \and Andrew Yuan} 
 \email{jeremy.booher@canterbury.ac.nz} 
 \address{School of Mathematics and Statistics\\
 University of Canterbury\\
   Christchurch, 8140, New Zealand}
   
   \email{pries@math.colostate.edu}
   \address{Department of Mathematics\\
   Colorado State University\\
   Fort Collins, Colorado 80523}

\begin{document}

\begin{abstract} 
Suppose $X$ is a smooth projective connected curve defined over an algebraically closed field of characteristic $p>0$
and $B \subset X$ is a finite, possibly empty, set of points.
Booher and Cais determined a lower bound for the $a$-number of a 
$\ZZ/p \ZZ$-cover of $X$ with branch locus $B$.
For odd primes $p$, in most cases it is not known if this lower bound is realized.
In this note, when $X$ is ordinary, 
we use formal patching to reduce that question to a computational question about 
$a$-numbers of $\ZZ/p\ZZ$-covers of the affine line.
As an application, when $p=3$ or $p=5$, for any ordinary curve $X$ and any choice of $B$, 
we prove that the lower bound is realized for Artin-Schreier covers of $X$ with branch locus $B$. 
\end{abstract}


\maketitle

%11G20  Number theory, Arithmetic Geometry, Curves over finite and local fields\\
%11T06 NT, finite fields and commutative rings, polynomials over finite fields\\
%14D15 Algebraic Geometry, Families/Fibrations, formal methods/deformations\\
%14H40 Algebraic Geometry, curves, Jacobians, Prym varieties\\
%15A04 Linear and multilinear algebra; matrix theory, Linear and semilinear transformations\\

%Secondary
%11C08 NT, polynomials and matrices, polynomials in NT
%12Y05 Field theory and polynomials, computational aspects
%14G17 Algebraic Geometry, arithmetic problems, Positive characteristic ground fields\\
%14H30  Algebraic Geometry, curves, Coverings/fundamental group\\
%15B33 Matrices over special rings (quaternions, finite fields, etc.\\

\section{Introduction}

Let $k$ be an algebraically closed field of characteristic $p>0$, and let $X$ be a smooth projective connected curve of genus $g_X$ defined over $k$.  The Jacobian ${\rm Jac}(X)$ of $X$ is an Abelian variety of dimension $g_X$.  The {\it $a$-number} of $X$ is an invariant of the $p$-torsion group scheme ${\rm Jac}(X)[p]$ of ${\rm Jac}(X)$.
It is defined to be 
\[a_X:={\rm dim}_k {\rm Hom}(\alpha_p, {\rm Jac}(X)[p]),\]
where $\alpha_p$ is the kernel of Frobenius on the additive group scheme $\GG_a$ and the homomorphisms are in the category of $k$-group schemes.

Now let $B \subset X(k)$ be a finite set of points. Suppose $\pi:Y \to X$ is a $\ZZ/p\ZZ$-cover branched at $B$.  
The ramification invariants of $\pi$ are the multi-set of values $D=\{d_Q\}_{Q \in B}$
where $d_Q$ is the jump in the ramification filtration (in the lower numbering) above $Q$.
The $a$-number of $Y$ is not generally determined from $a_X$, $B$, and $D$.
However, Booher and Cais determined upper and lower bounds on the $a$-number of a $\ZZ/p \ZZ$-cover $Y$ of $X$ with branch locus $B$ in terms of $a_X$ and the ramification invariants $D$ \cite{bc18}.  In particular, they prove that $a_Y \geq  L( D )$, where
\begin{equation} \label{eq:lowerbound}
 L( D ):=  \max_{1 \leq j \leq p-1} \sum_{Q\in B} \sum_{i=j}^{p-1} \left( \left\lfloor\frac{id_Q}{p}\right\rfloor - \left\lfloor\frac{id_Q}{p} - \left(1-\frac{1}{p}\right)\frac{jd_Q}{p}\right\rfloor\right).
\end{equation}
 We note that the lower bound $L(D)$ depends only on the multi-set of positive integers $\{d_Q\}_{Q \in B}$ and $p$ and does not depend on $X$, $a_X$, or the locations of the branch points.  
In Corollary~\ref{cor:maximized}, we prove that the maximum is achieved for $j = \frac{p-1}{2}$ 
in the right side of \eqref{eq:lowerbound} (or $j=1$ when $p = 2$).

In this paper, we study whether the lower bound is sharp; 
in other words, the goal is to determine whether there exists a cover $\pi : Y \to X$ with branch locus $B$ and ramification invariants $D$ such that $a_Y = L( D )$.  In the first main result, Theorem~\ref{thm:reduction}, 
when $X$ is ordinary,
we reduce this question to the question of whether the lower bound occurs for
$\ZZ/p\ZZ$-covers of the projective line $\PP^1$ 
branched at only one point.  In the second main result, Theorem~\ref{thm:basiccovers}, 
we give a positive answer to this latter question for small primes $p$.

This question fits into a broader program of understanding the $p$-torsion group schemes ${\rm Jac}(Y)[p]$ that occur for Artin-Schreier covers $\pi : Y \to X$ in terms of ${\rm Jac}(X)[p]$ and the ramification invariants of $\pi$.  
Since ${\rm Jac}(X)[p]$ is a group scheme of rank $p^{2 g_X}$,
the first result in this program is the Riemann-Hurwitz formula for wildly ramified covers, 
which determines the genus of $Y$ from the genus of $X$ and the ramification invariants.

The second result in this program is the Deuring-Shafarevich formula.  
The $p$-rank of $X$ is the integer $f_X$ such that $\#{\rm Jac}(X)[p](k) = p^{f_X}$; 
it is known that $0 \leq f_X \leq g$.
Equivalently, $f_X={\rm dim}_{\FF_p} {\rm Hom}(\mu_p, {\rm Jac}(X))$,
where $\mu_p$ is the kernel of Frobenius on $\GG_m$ and the homomorphisms are taken in the category of $k$-group schemes.  
One can check that 
$1 \leq a_X +f_X \leq g_X$. 
The curve $X$ is \emph{ordinary} if $f_X=g$ or, equivalently, if $a_X=0$.  
The Deuring-Shafarevich formula 
determines the $p$-rank of $Y$ from the $p$-rank of $X$ and the ramification invariants. 

Beyond these two results, little is known about the $p$-torsion group schemes ${\rm Jac}(Y)[p]$ that occur for Artin-Schreier covers $Y$ of $X$; see, for example, \cite{ElkinPries13} for hyperelliptic curves in characteristic 2.  
One reason for this is that there are many $\ZZ/p\ZZ$-covers $\pi : Y \to X$ in characteristic $p$.\footnote{ 
When $B$ is empty, there are $(p^{f_X}-1)/(p-1)$ unramified Galois covers $Y \to X$ with group $\ZZ/p \ZZ$,
because the pro-$p$ algebraic fundamental group $\pi_1^p(X)$
is a free pro-$p$ group on $f_X$ generators \cite{shafarevitch47}.  
When $B$ is non-empty, $\pi_1(X-B)$ is infinitely pro-finitely generated \cite{shafarevitch47}.}

The $a$-number can be computed from the action of the Cartier operator on the holomorphic $1$-forms of $Y$.
It depends in an intrinsic way on the Artin-Schreier equation $y^p - y = f$ with $f \in k(X)$
and thus is not determined from the branch locus and the ramification invariants.
By proving that the lower bound for the $a$-number 
can be realized, we provide valuable information about the structure of ${\rm Jac}(Y)[p]$ for a generic choice of $f$, 
see Remark~\ref{Rmoduli}.

We now set up some notation needed to state our results.

\begin{defn} \label{def:ad}
For $d \in \NN$ with $p \nmid d$, let $a(d)$ be the minimal $a$-number for a curve $Y$ over $k$
which admits a $\ZZ/p \ZZ$-cover of $\PP^1$ branched at a single point and having ramification invariant $d$.  
\end{defn}

Note that $a(d)$ depends on $p$.  
For $d \in \NN$ with $p \nmid d$, for a single branch point with ramification invariant $d$, the lower bound \eqref{eq:lowerbound} is
\[
 L(\{d\}) =  \max_{1 \leq j \leq p-1}  \sum_{i=j}^{p-1} \left( \left\lfloor\frac{id}{p}\right\rfloor - \left\lfloor\frac{id}{p} - \left(1-\frac{1}{p}\right)\frac{jd}{p}\right\rfloor\right).
\]

In our first main result, we use formal patching to prove the following reduction:

\begin{thm} \label{thm:reduction}
Let $X$ be an ordinary curve of genus $g_X$.
Suppose $B \subset X(k)$ has cardinality $r$, 
where $r>0$.
For each $Q\in B$, let $d_Q$ be a prime-to-$p$ positive integer and let $D=\{d_Q\}_{Q \in B}$.
Assume that for all $Q \in B$
\begin{equation} \label{Eassumegeneral}
a(d_Q) = L( \{d_Q\}).
\end{equation}
Then the lower bound \eqref{eq:lowerbound} is sharp; i.e., there exists a $\ZZ/p \ZZ$-cover $Y \to X$ with branch locus 
$B$ and ramification invariants $D$ such that $a_Y = L( D)$.
\end{thm}

In order to apply Theorem \ref{thm:reduction}, 
we need to verify the existence of covers $\pi : Y \to \PP^1$ branched at a single point such that $a_Y=L(\{d\})$. 
By Artin-Schreier theory, given a polynomial $f \in k[x]$, the affine equation $y^p-y=f$ defines 
a $\ZZ/p\ZZ$-cover $Y_f \to \PP^1$ branched only over infinity.  Let $d={\rm deg}(f)$.  
If $p \nmid d$, then $d$ is the ramification invariant above infinity, $Y_f$ has genus $g=(d-1)(p-1)/2$, and $Y_f$ has $p$-rank $0$.  From \eqref{eq:lowerbound}, it is immediate that $a_{Y_f} \geq L(\{d\})$.

\begin{thm} \label{thm:basiccovers}
When $p=3$ or $p=5$, for any positive integer $d$ with $p \nmid d$, there exists a polynomial $f \in \FF_p[x]$ of degree $d$ such that the Artin-Schreier curve $Y_f$
has $a$-number $L(\{d\})$.  In other words,
\[
a(d) = \max_{1 \leq j \leq p-1} \sum_{i=j}^{p-1} \left( \left\lfloor\frac{id}{p}\right\rfloor - \left\lfloor\frac{id}{p} - \left(1-\frac{1}{p}\right)\frac{jd}{p}\right\rfloor\right).
\]
\end{thm}

\begin{cor} \label{cor:sharp}
If $p = 2,3,5$, the assumption in \eqref{Eassumegeneral} is true and hence the lower bound \eqref{eq:lowerbound} 
for the $a$-number is sharp for $\ZZ/p \ZZ$-covers of any ordinary curve $X$, branched at any finite, non-empty set $B$ of points, 
and for all possible ramification invariants $\{d_Q\}_{Q \in B}$.
\end{cor}

\begin{remark} \label{rmk:p2}
\begin{enumerate}
\item When $p=2$ and $X$ is ordinary, 
Booher and Cais's lower and upper bounds coincide \cite[Remark 6.30]{bc18}. 
 This recovers the main result of \cite{volochChar2}.  
 Thus the $p=2$ case of Corollary~\ref{cor:sharp} was already known.  

\item It was also already known that the lower bound occurs when $X$ is ordinary and $d_Q \mid (p-1)$ for each $Q \in B$.  In that case, \cite[Corollary 1.4] {bc18} shows that the $a$-number of any $\ZZ/p\ZZ$-cover of $X$ branched over $B$ with ramification invariants $D$ realizes the lower bound.  This extends a result of Farnell and Pries in the case that $X \simeq \PP^1$ \cite{fp13}. 

\item We conjecture that Corollary~\ref{cor:sharp} holds for any prime $p$.  By Theorem~\ref{thm:reduction}, the essential case is for covers of the affine line.  From experiments, 
we expect that for a randomly chosen polynomial $f \in \FF_p[x]$ of degree $d$, 
the probability that the Artin-Schreier curve 
$Y_f:y^p-y=f$ has $a$-number equal to $L({d})$ is approximately $\frac{p-1}{p}$.
Tables \ref{fig:317}, \ref{fig:511}, \ref{fig:712}, and \ref{fig:117} show the distribution of $a$-numbers for the curves $Y_f$ for many randomly chosen polynomials $f$ for particular choices of $p$ and $d$, 
giving some evidence for this expectation.  
%For the chosen $(p,d)=(3,11)$, $(5,11)$, $(7,12)$, and $(11,7)$, the lower bounds are $8$, $10$, $18$, and $18$ respectively, and we see covers with minimal $a$-number.  

\item We give evidence that the lower bound \eqref{eq:lowerbound} may not be sharp when $X$ is non-ordinary.  
For example, when $p=3$, let $X$ be the hyperelliptic curve $y^2 = x^{10}+1$ over $\FF_3$; then $X$ has $a$-number $2$.  Suppose $Y \to X$ is a $\ZZ/3 \ZZ$-cover branched only over the point $P = (0,-1)$ and having ramification invariant $d=14$.  This cover is given by adjoining a root of $z^3 - z =f$ to $k(X)$ for some rational function
$f$ on $X$ regular away from $P$ such that the local expansion, with respect to a fixed uniformizer $t$ at $P$, 
is $t^{-14} + O(t^{-13})$.
Table~\ref{fig:d14} shows the number of such $f$ in $\FF_3(X)$ having a given $a$-number; 
the lower bound of $6$ for $a_Y$ does not occur for any such $f$.  It is conceivable, but we do not expect this to happen, that some cover defined over an extension of $\FF_3$ would have $a$-number $6$.
\end{enumerate}
\end{remark}

\begin{table}[hb]
\centering
\begin{minipage}{0.48\textwidth}
\centering
 \begin{tabular}{| c | c | c | c| }
 \hline
 $a_X$ & 8 & 9 & 10 \\
\hline
Number & 6633 & 2988 & 379 \\
 \hline
\end{tabular}
\caption{$a$-numbers of $10000$ covers of $\PP^1$, $p=3$ and $d=17$, lower bound $8$.}
 \label{fig:317}
\end{minipage}
\hfill
\begin{minipage}{0.48\textwidth}
\centering
 \begin{tabular}{| c | c | c| c | c| c|}
 \hline
 $a_X$ & 10 & 11 & 12 & 13 & 14\\
 \hline
 Number & $8021$ & $1901$ & $64$ & $10$ & $4$ \\
 \hline
\end{tabular}  
\caption{$a$-numbers of $10000$ covers of $\PP^1$, $p=5$ and $d=11$, lower bound $10$.}
\label{fig:511}
\end{minipage}
\end{table}

\begin{table}[hb]
\centering
\begin{minipage}{0.48\textwidth}
\centering
 \begin{tabular}{| c | c | c| c|}
 \hline
 $a_X$ & 18 & 19 & 20 \\
 \hline
 Number & $861$ & $138$ & $1$\\
 \hline
\end{tabular}
\caption{$a$-numbers of $1000$ covers of $\PP^1$, $p=7$ and $d=12$, lower bound 18.}
\label{fig:712}
\end{minipage}
\hfill
\begin{minipage}{0.48\textwidth}
\centering
 \begin{tabular}{| c | c | c| c|}
 \hline
 $a_X$ & 18 & 19 & 20 \\
 \hline
 Number & $894$ & $105$ & $1$\\
 \hline
\end{tabular} 
\caption{$a$-numbers of $1000$ covers of $\PP^1$, $p=11$ and $d=7$, lower bound 18.}
\label{fig:117}
\end{minipage}
\end{table}

%\begin{table} 
%\centering
% \begin{tabular}{| c| c | c | c | c | c | c |}
% \hline
% $a_X$ & 9 & 10 & 11 & 12 & 13 & 14\\
% \hline
% Number & 0 & $863$ & $111$ & $26$ & $0$ & $1$\\
% \hline
%\end{tabular}
%\caption{$a$-numbers of $1001$ covers of $y^2 = x^3 -x$, with $p=7$ and $d=6$, lower bound $9$.}
%\label{fig:d6}
%\end{table}

\begin{table} [ht]
\centering
 \begin{tabular}{| c| c | c | c | c |}
 \hline
 $a_X$ & 6 & 7 & 8 & 9 \\
 \hline
 Number & 0 & $39366$ & $13122$ & $6561$ \\
 \hline 
\end{tabular}
\caption{$a$-numbers of covers of $y^2 = x^{10}+1$ over $\FF_3$, with $d=14$, lower bound $6$.}
\label{fig:d14}
\end{table}
\begin{remark} \label{Rmoduli}
For a fixed number of branch points and fixed ramification invariants, there is a
moduli space of Artin-Schreier covers of the projective line with that ramification data; that moduli space 
is irreducible by \cite[Theorem 1.1]{PZ}.
Realizing the lower bound for covers with ramification invariants $D$ computes the $a$-number of the universal curve over the component $C_D$ with ramification invariants $D$.   The moduli space of Artin-Schreier covers has a stratification by $a$-number; realizing the lower bound shows that the covers with $a$-number $L(D)$ are dense in $C_D$.
\end{remark}

\subsection{Outline of the paper}

In \S\ref{sec:facts}, we recall some facts about the $a$-number of curves and the lower bound on the $a$-number of an Artin-Schreier cover obtained in \cite{bc18}.  
In Corollaries~\ref{cor:maximized} and \ref{cor:explicitlower}, 
we prove a combinatorial simplification of the lower bound in 
\eqref{eq:lowerbound} which is easier to work with.

In \S\ref{sec:patching}, we prove Theorem~\ref{thm:reduction}.  The strategy is to inductively construct the cover using formal patching.  Because of this argument, we can reduce to studying the existence of Artin-Schreier covers of 
$\PP^1$ branched only at infinity with minimal $a$-number.  
In \S\ref{sec:calculating}, we explain how to use the arguments of \cite{bc18} to give a method of computing the $a$-number of such a cover given an explicit description of the Artin-Schreier extension.  

In \S\ref{sec:char3}, in characteristic three, for any positive integer $d$ with $3 \nmid d$, we construct $\ZZ/3 \ZZ$-covers of $\PP^1$ branched only at infinity with ramification invariant $d$ and with $a$-number $L(\{d\})$.
This establishes the $p=3$ case of Theorem~\ref{thm:basiccovers}.  

In \S\ref{sec:oxford} and \S\ref{sec:boston}, in characteristic $5$, 
for any positive integer $d$ with $ 5 \nmid d$, we give different families of $\ZZ/5 \ZZ$-covers of $\PP^1$ 
branched only at infinity with ramification invariant $d$ and with $a$-number $L(\{d\})$.  These families were discovered experimentally using the computer algebra system MAGMA \cite{magma}.  
In \S\ref{sec:oxford}, we focus on families $y^5 - y =f$ where $f$ is a binomial of degree $d$; for many congruence classes of $d$ modulo $25$, we can find such a family with minimal $a$-number.  In contrast, the families considered in \S\ref{sec:boston} are of the form $y^5 - y = f$ where $f$ is a trinomial of degree $d$; there are four families with minimal $a$-number depending on $d$ modulo $5$.  

\begin{remark}
\begin{enumerate}
\item The methods of this paper still apply when $p>5$, but the analysis to compute the $a$-number becomes increasingly complicated.  The difficulty is in finding examples which are generic enough to have the minimal $a$-number but still
simple enough to analyze.  

\item The methods of this paper cannot be used to show the upper bound for the $a$-number of a $\ZZ/p\ZZ$-cover is sharp.

\item The methods used in Theorem~\ref{thm:reduction} also apply to other invariants of ${\rm Jac}(Y)$ including the Newton polygon of its $p$-divisible group, see \cite{booherpriesNP}, or the Ekedahl-Oort type of its $p$-torsion group scheme.  

\end{enumerate}
\end{remark}

\subsection{Acknowledgements}
This paper originated from work of Booher and Pries on realizing curves with minimal $a$-numbers.
The material in Sections \S\ref{sec:oxford} (resp.\ \S\ref{sec:boston})
was discovered by high-school students at the PROMYS program in Oxford (resp.\ Boston).
We would like to thank the PROMYS program for running these research lab projects, 
with support from the PROMYS Foundation and the Clay Mathematics Institute.  
We also thank Kevin Chang for mentoring the group at PROMYS in Boston, Bryden Cais for many helpful conversations, and the referees for their useful comments.

Booher was partially supported by the Marsden Fund Council administered by the Royal Society of New Zealand.
Pries was partially supported by NSF grant DMS-19-01819. 

\section{Facts about \texorpdfstring{$\ZZ/p\ZZ$}{Z/pZ}-covers and \texorpdfstring{$a$}{a}-numbers} \label{sec:facts}

Fix a prime $p$.  Let $k$ be an algebraically closed field of characteristic $p$; by a curve over $k$, we mean a smooth, connected, projective curve over $k$.

\subsection{Genus, \texorpdfstring{$p$}{p}-Rank, and \texorpdfstring{$a$}{a}-Number}

\begin{notation} \label{Ncover}
Let $\pi:Y \to X$ be a branched Galois cover of curves over $k$ with Galois group $\ZZ/p\ZZ$.  Let $B \subset X(k)$ be the branch locus. 
Since the ramification index divides the degree, $\pi$ is totally ramified above each point of $B$.
For $Q \in B$, let $d_Q$ denote the break in the ramification filtration at the unique ramification point above $Q$ (in the lower numbering).  Let $g_Y$ be the genus of $Y$ and $g_X$ be the genus of $X$.
\end{notation}

We refer to the indexed multi-set $D=\{d_Q\}_{Q \in B}$ as the \emph{ramification invariants} of $\pi$.  The ramification invariants $d_Q$ are positive integers that are prime-to-$p$.

\begin{fact} \label{TRH} (Riemann-Hurwitz formula \cite[IV, Prop.\ 4]{serreLF})
With notation as in Notation~\ref{Ncover}, 
\[2g_Y-2 = p(2g_X-2) + \sum_{Q \in B} (p-1)(d_Q+1).\]
\end{fact}

Let $f_X$ be the $p$-rank of $X$ and let $f_Y$ be the $p$-rank of $Y$.  By definition,
\begin{equation} \label{Dprank}
f_X={\rm dim}_{\FF_p}{\rm Hom}(\mu_p, {\rm Jac}(X)[p])
\end{equation}
where the homomorphisms are taken in the category of $k$-group schemes.

\begin{fact} \label{TDS} (Deuring-Shafarevich formula \cite[Theorem 4.2]{subrao})
With notation as in Notation \ref{Ncover}, then
\[f_Y-1 = p(f_X-1) + \#B(p-1).\]
\end{fact}

\begin{lem} \label{lem:ordinarycover}
If $X$ is ordinary and $\pi: Y \to X$ is an unramified $\ZZ/p \ZZ$-cover then $Y$ is also ordinary.
\end{lem}

\begin{proof}
This follows from the Deuring-Shafarevich formula.
\end{proof}

\begin{defn} \label{defn:a-number}
Let $X$ be a curve over $k$.  Let $\alpha_p$ be the kernel of Frobenius on $\Ga$.  The $a$-number of $X$ is defined to be
\begin{equation} \label{Edefanumber}
a_X := {\rm dim}_k {\rm Hom}(\alpha_p, {\rm Jac}(X)[p])
\end{equation}
where the homomorphisms are taken in the category of $k$-group schemes.
\end{defn}

It can equivalently be defined as the dimension of the kernel of the Cartier operator $\Cartier_X$ on the space of regular differentials on $X$. 

There is no analogue of the Riemann-Hurwitz or Deuring-Shafarevich formulas for the $a$-number as $a_Y$ is not determined by $X$, $B$, and $D$.

\begin{remark}
Suppose $X$ is a singular curve, whose only singularities are ordinary double points, possibly of non-compact type.  Then it is still possible to define the genus, $p$-rank and $a$-number of $X$.  Let $\tilde{X}$ be the normalization of $X$.  The Jacobian ${\rm Jac}(X)$ is a semi-abelian variety of dimension $g$.  By \cite[9.2.8]{BLR}, there is an exact sequence
\[0 \to T \to {\rm Jac}(X) \to {\rm Jac}(\tilde{X}) \to 0,\]
for some torus $T$.
Let $\epsilon$ be the rank of $T$.
Then the (geometric) genus, $p$-rank, and $a$-numbers are related by:
\begin{equation} \label{Erelate}
g_X = \epsilon + g_{\tilde{X}}, \ f_X = \epsilon + f_{\tilde{X}}, \ a_X = a_{\tilde{X}}.
\end{equation}
More formally, \eqref{Dprank} and \eqref{Edefanumber} are valid definitions for the $p$-rank and $a$-number of a 
curve whose only singularities are ordinary double points, and these definitions are compatible with 
the equalities in \eqref{Erelate}. 
\end{remark}

\subsection{Bounds on the \texorpdfstring{$a$}{a}-number } \label{sec:boundsona}

In this section, we describe a lower bound for the $a$-number of a branched $\ZZ/p\ZZ$-cover of curves.
The following notation will be useful.

\begin{defn} \label{Dnewlower}
Let $d$ be a positive prime-to-$p$ integer.
Let $j$ be an integer with $0 \leq j \leq p-1$.
For an integer $i$ with $j \leq i \leq p-1$, 
we define:
\[\tau_{i,j} = id - (1-\frac{1}{p})dj;\]
\[A_{i,j} = \{n :  p |n \text{ and }  \tau_{i,j} < n \leq i d \};\]
\begin{equation} \label{DLij}
L_{i,j}(\{d\}) = 
\left\lfloor\frac{id}{p}\right\rfloor - \left\lfloor\frac{\tau_{i,j}}{p}\right\rfloor=
\left\lfloor\frac{id}{p}\right\rfloor - \left\lfloor\frac{id}{p} - \left(1-\frac{1}{p}\right) \frac{jd}{p}\right\rfloor;
\end{equation}
and
\[E_{i,j} = \{n : p \mid n \text{ and }  \tau_{i,j} < n \leq \tau_{i,j-1}\}.\]
\end{defn}

Notice that $L_{i,j}(\{d\}) = |A_{i,j}|$ because $\left\lfloor\frac{id}{p}\right\rfloor$ (resp.\ 
$\left\lfloor\frac{\tau_{i,j}}{p}\right\rfloor$)
is the number of positive  multiples of $p$ less than or equal to $id$ (resp.\ $\tau_{i,j}$).
When $(i,j) \neq (0,0)$, we note that $\tau_{i,j} >0$ and $\tau_{i,j}$ is not an integer that is divisible by $p$.

\begin{defn} \label{DnewdefL}
Let $j$ be an integer with $0 \leq j \leq p-1$.
For a positive prime-to-$p$ integer $d$, define
\[L_j(\{d\})= \sum_{i=j}^{p-1} L_{i,j}(\{d \}) = \sum_{i=j}^{p-1} \left( \left\lfloor\frac{id}{p}\right\rfloor - \left\lfloor\frac{id}{p} - \left(1-\frac{1}{p}\right) \frac{jd}{p}\right\rfloor \right)
. \]
\end{defn}

\begin{fact}[{\hspace{1sp}\cite[Theorem 1.1]{bc18}}] \label{fact:lowerbound}
With notation as in Notation \ref{Ncover} and Definitions \ref{Dnewlower} and \ref{DnewdefL}, then
\begin{equation} \label{equation:lowerbound}
a_Y \geq L(D) := \max_{1 \leq j \leq p-1} \sum_{Q \in B} L_j(\{d_Q\}).
\end{equation}

\end{fact}

For our applications, we need to show that $L(D)$ 
is ``additive'' in $B$; in other words, that we can reverse the ordering of taking the maximum and summing 
over $Q \in B$.
Corollary~\ref{cor:maximized} gives an alternate expression for $L(D)$ which makes this property clear. 
The proof is intricate so  
we begin with a few preliminary arguments.

\begin{lem} \label{lem:sets}
If $p$ is odd and $(p+1)/2  \leq j \leq p-1$, then 
\[
\bigcup_{i=j+1}^{p-1} E_{i,j+1} \subset A_{j,j}
\]
and the union on the left is a disjoint union.
\end{lem}

\begin{proof}
Recall that $\tau_{i,j} = id - (1-\frac{1}{p})dj$, and notice that
$\tau_{i,j} \leq \tau_{i,j-1} \leq \tau_{i+1,j}$.  Then we see that
\[
\tau_{j,j} \leq \underbrace{\tau_{j+1,j+1} \leq \tau_{j+1,j} }_{E_{j+1,j+1}} \leq \underbrace{ \tau_{j+2,j+1} \leq \tau_{j+2,j}}_{E_{j+2,j+1}} \leq \ldots \leq \underbrace{ \tau_{p-1,j+1} \leq \tau_{p-1,j}}_{E_{p-1,j+1}}.
\]
As $(p+1)/2 \leq j$, we see that $\tau_{p-1,j} \leq p-1 - (1-\frac{1}{p}) (p+1)/2 \leq jd$.
Now $\tau_{j,j}$ and $jd$ are the bounds in the definition of $A_{j,j}$, 
and the intermediate terms are the bounds in the definition of $E_{i,j+1}$ as indicated.  This completes the proof. 
\end{proof}

\begin{lem} \label{lem:decreasing}
If  $p$ is odd, $(p+1)/2 \leq j \leq p-2$ and $d$ is a positive prime-to-$p$ integer, then
\[
L_j(\{d\}) \geq L_{j+1}(\{d\}).
\]
\end{lem}

\begin{proof}
By Definition~\ref{Dnewlower}, $A_{i,j+1} = A_{i,j} \cup E_{i,j+1}$ and this union is disjoint.  
By Lemma~\ref{lem:sets}, \begin{align*}
L_j(\{d\}) & = |A_{j,j}|  + \sum_{i=j+1}^{p-1} |A_{i,j}| \\
& = |A_{j,j}| + \sum_{i=j+1}^{p-1} \left( |A_{i,j+1}| - |E_{i,j+1}| \right)\\
&\geq \sum_{i=j+1}^{p-1} |A_{i,j+1}| = %|A_{j+1}| =
L_{j+1}(\{d\}). \qedhere
\end{align*}
\end{proof}

Our next goal is to prove that $L_j(\{d\}) = L_{p-j} (\{ d\})$.  

\begin{defn}
Suppose $j < p-j$.
 We set
\begin{align*}
S_1 & := \sum_{i=j}^{p-1-j} \left \lfloor \frac{id}{p} \right \rfloor;\\
S_2 &:=  \sum_{i=j}^{p-1}\left \lfloor \frac{\tau_{i,j}}{p} \right \rfloor = \sum_{i=j}^{p-1}   \left\lfloor\frac{id}{p} - \left(1-\frac{1}{p}\right)\frac{jd}{p}\right\rfloor; \ {\rm and}\\
S_3 &:= \sum_{i=p-j}^{p-1} \left \lfloor \frac{\tau_{i,p-j}}{p} \right \rfloor =  \sum_{i=p-j}^{p-1} \left\lfloor\frac{id}{p} - \left(1-\frac{1}{p}\right)\frac{(p-j)d}{p}\right\rfloor.
\end{align*}
\end{defn}

\begin{lem} \label{lem:claim1}
For $j < p-j$, we have $L_j(\{d\}) - L_{p-j}(\{d\}) = S_1 - S_2 + S_3$.
\end{lem} 

\begin{proof}
We compute that
\begin{eqnarray*} 
L_j(\{d\}) - L_{p-j}(\{d\}) & = & \sum_{i=j}^{p-1} L_{i,j}(\{d\}) - \sum_{i=p-j}^{p-1}  L_{i,j}(\{d\}) \\
& = &  \sum_{i=j}^{p-1} \left(\left\lfloor\frac{id}{p}\right\rfloor - \left\lfloor \frac{ \tau_{i,j}}{p} \right\rfloor \right) - \sum_{i=p-j}^{p-1} \left(\left\lfloor\frac{id}{p}\right\rfloor - \left\lfloor \frac{\tau_{i,p-j}}{p} \right\rfloor \right)\\
& = & \sum_{i=j}^{p-j-1} \left\lfloor\frac{id}{p}\right\rfloor  - \sum_{i=j}^{p-1} \left\lfloor \frac{ \tau_{i,j}}{p} \right\rfloor + \sum_{i=p-j}^{p-1} \left\lfloor \frac{\tau_{i,p-j}}{p} \right\rfloor  \\
& = & S_1 - S_2 + S_3.  
\end{eqnarray*}
\end{proof}

\begin{lem} \label{lem:symmetry}
If $d$ is a positive prime-to-$p$ integer and $0 \leq j \leq p-1$, then $L_j(\{d\}) = L_{p-j} (\{ d\})$.  
\end{lem}

\begin{proof}
Without loss of generality, we can suppose that $j < p-j$.
By Lemma~\ref{lem:claim1}, it suffices to show that $S_1-S_2+S_3 =0$.  We do so by 
simplifying the expressions for $S_1$ and $S_3$.  The case $p=2$ is clear by inspection, so we assume that $p \neq 2$.  
In $S_1$, pairing $i$ with $p-i$ and using that, for $0 < i < p$,
\begin{equation} \label{eq:elementarysum}
\left \lfloor \frac{i d}{p} \right \rfloor + \left \lfloor \frac{(p-i) d}{p} \right \rfloor = d-1,
\end{equation}
we compute that
\[
S_1 = (d-1)(p-1-2j)/2 + \left \lfloor \frac{j d}{p} \right \rfloor.
\]
Next, notice that
\[
S_3 =  \sum_{i=p-j}^{p-1} \left \lfloor \frac{(i +1-p)d}{p}  -  \left(1-\frac{1}{p}\right)\frac{-jd}{p} \right \rfloor.
\]
Let $t=p-1-i$; when $i=p-1$, then $t=0$ and when $i=p-j$, then 
$t=j-1$.  Re-indexing the sum (and using that $\lfloor - a \rfloor = -1 - \lfloor a \rfloor$ when $a$ is not an integer), we obtain
\[
S_3 = \sum_{t=0}^{j-1} \left \lfloor \frac{-t d}{p} -  \left(1-\frac{1}{p}\right)\frac{-jd}{p} \right \rfloor = -j - \sum_{t=0}^{j-1} \left \lfloor \frac{t d}{p} -  \left(1-\frac{1}{p}\right)\frac{jd}{p} \right \rfloor.
\]
Therefore, we see that
\begin{equation*} 
S_2 - S_3 -j = \sum_{i=0}^{p-1} \left\lfloor\frac{id}{p} - \left(1-\frac{1}{p}\right)\frac{jd}{p}\right\rfloor = \sum_{i=0}^{p-1} \left\lfloor\frac{id}{p} \right \rfloor + \sum_{i=0}^{p-1} \left \lfloor \frac{( id \operatorname{mod}p )}{p}  - \left(1-\frac{1}{p}\right)\frac{jd}{p} \right \rfloor.
\end{equation*}
Here, as usual, $(id \operatorname{mod}p)$ denotes the least non-negative remainder when $id$ is divided by $p$.
But $id$ is a complete set of residues modulo $p$ for $i \in \{0,1,\ldots, p-1\}$, so we may re-index the second sum and obtain
\begin{equation}\label{eq:sum}
S_2 - S_3 -j = \sum_{i=0}^{p-1} \left\lfloor\frac{id}{p} \right \rfloor + \sum_{t=0}^{p-1} \left \lfloor \frac{t}{p}  - \left(1-\frac{1}{p}\right)\frac{jd}{p} \right \rfloor.
\end{equation}
Now recall Hermite's identity: for an integer $n$ and real number $x$, 
\[
\lfloor nx \rfloor = \sum_{i=0}^{n-1} \lfloor x + i/n \rfloor.
\]
Applying this to \eqref{eq:sum} with $x = -\left(1-\frac{1}{p}\right)\frac{jd}{p}$, we obtain that
\[
S_2 - S_3 = j + \sum_{i=0}^{p-1} \left(\left\lfloor\frac{id}{p} \right \rfloor + \left \lfloor \frac{(1-p) jd}{p} \right \rfloor \right).
\]
The remaining sum can be evaluated using \eqref{eq:elementarysum}, giving that
\[
S_2 - S_3 = j + \frac{(p-1)}{2} (d-1) - j d + \left \lfloor \frac{jd}{p} \right \rfloor.
\]
Therefore we see that $S_1 - S_2 + S_3 = 0$ as desired, so $L_j(\{d\}) = L_{p-j}(\{d\})$.
\end{proof}

\begin{cor} \label{cor:maximized}
With notation as above and $p$ odd:
\begin{enumerate}
\item $\displaystyle L(D) = \sum_{Q \in B} L_{\frac{p-1}{2}}(\{d_Q\})$;
\item $\displaystyle L(D) = \sum_{Q \in B}  L(\{d_Q\})$;
\item $\displaystyle L(D) =\sum_{Q \in B} \sum_{i=(p-1)/2}^{p-1} \left( \left\lfloor\frac{id_Q}{p}\right\rfloor - \left\lfloor\frac{id_Q}{p} - \left(1-\frac{1}{p}\right)\frac{(p-1)d_Q}{2p}\right\rfloor\right)$.
\end{enumerate}
\end{cor}

The corollary could equally well use $(p+1)/2$ instead of $(p-1)/2$.  Note that if $p=2$, by definition $\displaystyle L(D) = \sum_{Q \in B} L_1(\{d_Q\})$.

\begin{proof}
The proof of part (1) follows from Lemma~\ref{lem:decreasing} and Lemma~\ref{lem:symmetry}
because $L_j(\{d\})$ is maximized when $j=\frac{p-1}{2}$ (or when $j = \frac{p+1}{2}$) for any $d$.
The other parts are immediate by definition.
\end{proof}

Recall from Fact~\ref{fact:lowerbound} that $a_Y \geq L(D)$.  
To summarize, we have proven that:

\begin{cor} \label{cor:explicitlower}
With the notation of Notation~\ref{Ncover}, if $p$ is odd
$$a_Y \geq L(D) =\sum_{Q \in B} \sum_{i=(p-1)/2}^{p-1} \left( \left\lfloor\frac{id_Q}{p}\right\rfloor - \left\lfloor\frac{id_Q}{p} - \left(1-\frac{1}{p}\right)\frac{(p-1)d_Q}{2p}\right\rfloor\right).$$
\end{cor}

\section{An Inductive Approach Using Formal Patching} \label{sec:patching}

The ideas used in this section to study covers of curves in positive characteristic originate with 
Grothendieck's work on the \'etale fundamental group.
Our results are a straight-forward application of the theory of formal patching (see \cite{harbaterstevenson} for a good reference).  Given several Artin-Schreier covers of the projective line branched only at one point and having minimal $a$-number, we may use formal patching to combine these together to obtain $\ZZ/p\ZZ$-covers of an arbitrary curve with arbitrary branch locus which still have minimal $a$-number.

As before, fix a prime $p$ and an algebraically closed field $k$ 
of characteristic $p$, and a smooth projective connected curve $X$ over $k$.

We restate the theorem from the introduction.  Recall from Definition~\ref{def:ad} 
that $a(d)$ is the minimal $a$-number of a 
curve which is realizable as a $\ZZ/p\ZZ$-cover of $\PP^1$ branched at a single point and having ramification invariant $d$. 

\begin{thm} \label{thm:reduction2}
Let $X$ be an ordinary curve of genus $g_X$.
Suppose $B \subset X(k)$ has cardinality $r$, 
where $r>0$.
For each $Q\in B$, let $d_Q$ be a prime-to-$p$ positive integer and
assume that 
\begin{equation} \label{Eassumegeneralrestate}
a(d_Q)=L( \{d_Q\}).
\end{equation}
Then the lower bound \eqref{eq:lowerbound} is sharp; i.e., there exists a $\ZZ/p \ZZ$-cover $Y \to X$ 
with branch locus $B$ and ramification invariants $D$ such that $a_Y = L( D )$.
\end{thm}

\begin{proof}
If $X \simeq \PP^1$ and $r=1$, then the result is vacuously true by the hypothesis in \eqref{Eassumegeneralrestate}.
So we assume that $r \geq 2$ if $X \simeq \PP^1$. 

For each $Q \in B$, by the hypothesis $a(d_Q) =L( \{d_Q\})$, there exists a 
$\ZZ/p\ZZ$-cover $\pi_Q: Y_Q \to \PP^1$ with $a$-number $a_Q:=L( \{d_Q\})$ and branched only at $\infty$ where it has ramification invariant $d_Q$.
Let $g_Q$ be the genus of $Y_Q$; by Fact \ref{TRH}, $g_Q = (p-1)(d_Q-1)/2$.

Next we define an unramified $\ZZ/p\ZZ$-cover $\pi': Y' \to X$.
If $g_X =0$, this can be done by letting $Y' = {\rm Ind}_{\{0\}}^{\ZZ/p\ZZ} X$, 
so that $Y'$ is a disconnected curve, having $p$ components, 
each of which is isomorphic to $X$.  
If $g_X >0$, since $X$ is ordinary, there exists an unramified $\ZZ/p\ZZ$-cover $\pi': Y' \to X$ 
such that $Y'$ is connected.\footnote{When $g_X >0$, it is also possible to work with 
a cover such that $Y'$ is disconnected; we choose to work with a connected cover since it provides additional control that might be useful in future applications.} In this case, $Y'$ is ordinary by Lemma \ref{lem:ordinarycover};
also, by Fact \ref{TRH}, the genus of $Y'$ is
$g_{Y'} = g_X + (p-1)(g_X -1)$.

We follow the strategy in \cite[Section 2.3]{booherpriesNP}
to build an Artin-Schreier cover
$\pi_\circ: Y_\circ \to X_\circ$ of singular curves, depending on the data of 
$\pi': Y' \to X$ and $\pi_Q: Y_Q \to \PP^1$ for each $Q \in B$. 
We build the singular curve $X_\circ$ by attaching a projective line to $X$ at each $Q \in B$, 
by identifying the point $0$ on $\PP^1$ and the point $Q$ in an ordinary double point.
Next, we build a singular curve $Y_\circ$ whose components are $Y'$ and $Y_Q$ for $Q \in B$, formed 
by identifying the fiber of $\pi_Q$ above $0$ and the fiber of $\pi'$ above $Q$, in $p$ ordinary double points, 
in a Galois equivariant way, for each $Q \in B$.  
There is a $\ZZ/p \ZZ$-cover $\pi_\circ: Y_\circ \to X_\circ$.  

In \cite[Proposition 2.5]{booherpriesNP}, we use \cite[9.2.8]{BLR} to prove the following:
if $g_X =0$, then ${\rm Jac}(Y_\circ)$ 
is an extension of 
${\rm Jac}(X)^p \oplus \left (\bigoplus_{Q \in B} {\rm Jac}(Y_Q) \right) $ by a torus of rank $\epsilon=(r-1)(p-1)$;  
if $g_X >0$, then ${\rm Jac}(Y_\circ)$ is an extension of 
${\rm Jac}(Y') \oplus \left (\bigoplus_{Q \in B} {\rm Jac}(Y_Q) \right) $ by a torus of rank $\epsilon=r(p-1)$.  
Furthermore, in either case, ${\rm Jac}(Y_\circ)$
is a semi-abelian variety of dimension
\[g_{Y_\circ}=g_X+(p-1)(g_X+r-1) + \sum_{Q \in B}g_{Q}.\] 
The torus does not contribute to the $a$-number.  So, by construction, the $a$-number of $Y_\circ$ is
$a_{Y_\circ} = pa_X + \sum_{Q \in B} a_Q$ if $g_X = 0$ and is
$a_{Y_\circ} = a_{Y'} + \sum_{Q \in B} a_Q$ if $g_X >0$.
If $g_X=0$, then $pa_X=0$ and if $g_X >0$ then $a_{Y'}=0$ since $Y'$ is ordinary.
In either case, $a_{Y_\circ} =  \sum_{Q \in B} L( \{d_Q\})$.
By Corollary~\ref{cor:maximized},
the $a$-number of $Y_\circ$ is $L( D )$.

By the theory of formal patching \cite{harbaterstevenson}, 
the $\ZZ/p \ZZ$-cover $\pi_\circ:Y_\circ \to X_\circ$ has a deformation to a 
$\ZZ/p \ZZ$-cover $\tilde{\pi}:\tilde{Y} \to X$ such that $\tilde{Y}$ is smooth and connected,  
$\tilde{\pi}$ is branched only above $B$, and $\tilde{\pi}$ has ramification invariant $d_Q$ above each $Q \in B$; 
in particular, this implies that $Y$ and $\tilde{Y}$ have the same genus,
see \cite[Remark 2.5]{booherpriesNP} for details. 
Let $a_{\tilde{Y}}$ be the $a$-number of $\tilde{Y}$.
Then $a_{\tilde{Y}} \leq a_{Y_\circ}$ since the $a$-number can only increase under specialization.
Also $a_{\tilde{Y}} \geq a_{Y_\circ}$ since $a_{Y_\circ} = L( D )$,
which is the lower bound for the $a$-number of $\tilde{Y}$ by Corollary~\ref{cor:explicitlower}.
This completes the proof.
\end{proof}

\begin{remark}
In \cite{booherpriesNP}, we use 
the same strategy of formal patching to realize lower bounds for Newton polygons of Artin-Schreier covers.
This strategy can also be used for Ekedahl-Oort types.
\end{remark}

\section{Computing the \texorpdfstring{$a$}{a}-number} \label{sec:calculating}

The proof of Fact~\ref{fact:lowerbound} gives a method to calculate $a_Y$ using an explicit description of the cover
$\pi:Y \to X$.  We work over a perfect field $k$, not necessarily algebraically closed, as the examples we later work with are naturally defined over $\FF_p$.  The $a$-number does not change under a field extension and the method of calculation is also independent of the base field.

The method ultimately relies on the description of $a_Y$ as the dimension of the kernel of the Cartier operator $\Cartier_Y$ on $H^0(Y,\Omega^1_Y)$, but is considerably more explicit.  
The big idea in the proof is that while the sheaf-theoretic kernel of $\Cartier_Y$ is complicated, it sits inside a much simpler sheaf denoted by $\Gscr_p$.  Using local methods, 
one obtains skyscraper sheaves $M_i$ for $0 \leq i \leq p-1$ on $\PP^1$ supported at infinity and a linear transformation 
\[
\widetilde{g} : H^0(\PP^1,\Gscr_p) \to \bigoplus_{i=0}^{p-1} H^0(\PP^1,M_i)
\]
such that the kernel of $\tilde{g}$ is $ H^0(Y,\ker \Cartier_Y)  = H^0(\PP^1, \pi_* \ker \Cartier_Y) \subset H^0(\PP^1,\Gscr_p)$. %\cite[Definitions 5.6, 6.2 and Lemma 6.3]{bc18}.
In this paper, we specialize this method to the basic class of examples we need to study.

\begin{defn}
A \emph{basic Artin-Schreier cover} is 
a branched $\ZZ/p\ZZ$-Galois cover $\pi : Y \to \PP^1$ that is branched only at infinity.
In this situation, we call $Y$ a \emph{basic Artin-Schreier curve}. 
\end{defn}

Given a basic Artin-Schreier cover $\pi$, the 
corresponding extension of function fields is obtained by adjoining a root of $y^p -y - f$ to $k(x)$ 
for some polynomial $f\in k[x]$.
By Artin-Schreier theory, $f \not = \alpha^p - \alpha$ for any $\alpha \in k[x]$; furthermore,
without loss of generality, we can suppose that $f(x)$ 
contains no monomials of the form $c x^{p i}$.  
Conversely, such a polynomial $f$ determines a basic Artin-Schreier cover and thus
a basic Artin-Schreier curve, which we denote by $y^p-y=f$.
By \cite[III, 7.8(c)]{stich}, 
the ramification invariant for $\pi$ at infinity is ${\rm deg}(f)$.  

We fix a basic Artin-Schreier curve $Y$ and such a polynomial $f$, and let $d = \deg(f)$.  

\begin{remark}
The Riemann-Hurwitz formula implies that the genus of $Y$ is $(p-1)(d-1)/2$, and the Deuring-Shafarevich formula implies the $p$-rank of $Y$ is $0$.  
\end{remark}

\begin{defn} \label{def:ni}
For $0 \leq i \leq p-1$, define $n_{i} := \left \lceil \frac{(p-1-i) d}{p} \right \rceil -2$.
\end{defn}

\begin{fact}[{\hspace{1sp}\cite[Lemma 3.7]{bc18}}] \label{fact:regularity}
Every differential $\omega$ on $Y$ can be written as
$\omega = \displaystyle \sum_{i=0}^{p-1} \omega_i y^i $ where $\omega_0, \ldots, \omega_{p-1}$ are differentials on ${\mathbf P}^1$.
The differential $\omega$ is regular provided each $\omega_i$ is regular away from infinity, and $\ord_\infty(\omega_i) \geq -\left \lceil \frac{(p-1-i) d}{p} \right \rceil$.  More explicitly, writing $\omega_i = h_i(x) dx$, then $\omega$ is regular provided that $h_i(x)$ is a polynomial of degree at most $n_i$ for $0 \leq i \leq p-1$.
\end{fact}
 
We mostly work with the Cartier operator on $\PP^1$, which we denote $\Cartier$ for simplicity.  
The Cartier operator is a semi-linear operator satisfying
\begin{equation}
\Cartier\left( \sum_i a_i x^i \frac{dx}{x} \right) = \sum_i a_{pi}^{1/p} x^i \frac{dx}{x},
\end{equation}
where the coefficients $a_i$ of the series lie in $k$.  For a differential $\omega' = \sum_i b_i x^i dx$ on $\PP^1$, observe that $\Cartier(\omega') = 0$ if and only if $b_i =0$ for every $i \equiv -1 \pmod{p}$.

For a $\ZZ/p\ZZ$-cover $\pi : Y \to X$, \cite{bc18} introduces a sheaf $\Gscr_p$ on $X$ and an inclusion $H^0(Y,\Omega^1_Y) \into H^0(X,\Gscr_p)$ which can be used to control the $a$-number of $Y$.  We make this explicit when $X = \PP^1$.

\begin{lem} \label{lem:V}
There is a natural isomorphism between the vector space $H^0(\PP^1,\Gscr_p)$ from \cite[Definition 5.2]{bc18} 
and the vector space
\[
V_f:= \left\{\nu= ( \nu_0, \nu_1, \ldots, \nu_{p-1} ) : \nu_i = h_i(x) dx, \, \Cartier(\nu_i)=0, \text{ and } \deg(h_i) \leq n_i  \text{ for all } i \right\}.
\]
Equivalently, 
$\nu \in V_f$ if and only if $h_i(x) = \sum_{j=0}^{n_i} b_{i,j} x^j$ with $b_{i,j}=0$ if $j \equiv -1 \pmod{p}$,
for each $ 0 \leq i \leq p-1$.
\end{lem}

\begin{proof}
This is an immediate consequence of the definitions, Fact \ref{fact:regularity}, and the explicit description of the kernel of $\Cartier$ on $\PP^1$.
\end{proof}

Note that $V_f$ depends only on the degree of $f$.  We will often drop the subscript $f$ when 
the equation $y^p - y = f$ for a basic Artin-Schreier curve is clear from context.

\begin{remark} \label{rmk:dimV}
For later use, notice that the number of monomials in $\sum_{j=0}^{m} c_j x^j$ with $c_j =0$ when $ j\equiv -1 \pmod{p}$ is $m + 1 - \left \lfloor \frac{m +1}{p} \right \rfloor$.
Therefore
\[
\dim V_f = \sum_{i=0}^{p-1} \left(n_i+1 - \left \lfloor \frac{n_i+1}{p} \right \rfloor \right)= \sum_{i=0}^{p-1} \left( \left \lceil \frac{(p-1-i)d}{p} \right \rceil - \left \lceil \frac{(p-1-i)d}{p^2} \right \rceil \right).
\]
The last equality uses that $1 + \lfloor (x-1)/p \rfloor = \lceil x/p \rceil$.
\end{remark}

We also need a right
inverse to the Cartier operator on the projective line.

\begin{defn} \label{def:s}
 Define a function $s: \Omega^1_{\PP^1} \to \Omega^1_{\PP^1}$
 as follows; for $\omega = \sum_i c_i x^i dx$ with $c_i \in k$ define
 \[
  s \left(\sum_i c_i x^i \frac{dx}{x} \right) = \sum_j c_j^p x^{p j } \frac{dx}{x}.
 \]
\end{defn}

It is straightforward to check that $s$ is semi-linear and that $\Cartier(s(\omega)) = \omega$.

\begin{remark} \label{rmk:othercomposition}
The function $s$ is not a left inverse for $\Cartier$.  It is elementary to check that
\begin{equation} \label{eq:othercomposition}
s \left( \Cartier  \left( \sum_i a_i x^i \frac{dx}{x} \right)  \right) = \sum_{i \equiv 0 \pmod{p} } a_{i} x^{i} \frac{dx}{x}.
\end{equation}
\end{remark}

The bounds of Fact~\ref{fact:lowerbound} come from a detailed understanding of an inclusion of $H^0(Y,\ker \Cartier_Y)$ into $V_f = H^0(\PP^1,\Gscr_p)$ \cite[Definition 4.3]{bc18}.  Instead of describing this inclusion, we directly describe the image. 

\begin{defn}\label{def:omega}
For $\nu = (\nu_0,\nu_1,\ldots, \nu_{p-1}) \in V_f$, we define a differential 
$\gamma_f(\nu)$ on $Y$ as follows.
We let $\gamma_f(\nu)=\omega =\displaystyle \sum_{i=0}^{p-1} \omega_i y^i$ 
where $\omega_i$ are defined recursively as follows, starting with $i = p-1$:
\begin{itemize}
 \item Define $\omega_{p-1} = \nu_{p-1}$.
 \item Fix an integer $t$ satisfying $0 \leq t < p-1$ such that $\omega_j$ has already been defined for $t < j < p$.  Then define
 \begin{equation} \label{eq:reconstruct}
  \omega_t = \nu_t + s \left(\Cartier \left( -\sum_{j=t+1}^{p-1} \binom{j}{t} \omega_j (-f)^{j-t}  \right) \right) .
 \end{equation}
\end{itemize}
This recursively defines $\gamma_f(\nu)=\omega$.
\end{defn}

Note that the definition depends on the specific Artin-Schreier cover as it involves $f$.

\begin{defn} \label{def:gamma}
For $\nu \in V_f$, recall that $\gamma_f(\nu)$ is the differential $\omega$ constructed in Definition~\ref{def:omega}.   Set
\[
 W_f := \{ \nu \in V_f: \gamma_f(\nu) \text{ is regular}\}.
\]
\end{defn}

Both $W_f$ and $\gamma_f(\nu)$ depend on the basic Artin-Schreier cover $y^p-y =f$.  
We will drop the subscript $f$ when it is clear from context.  
It is straightforward to verify that $W_f$ is a vector space.

\begin{lem} \label{lem:calc-anumber}
Under the identification $V_f \simeq H^0(\PP^1,\Gscr_p)$ of $k$-vector spaces from Lemma~\ref{lem:V}, the injective map $H^0(\varphi) : H^0(Y,\ker \Cartier_Y) \into H^0(\PP^1,\Gscr_p)$ from \cite[Definition 4.3]{bc18}  is identified with the inclusion $W_f \subset V_f$. 
  Hence the $a$-number of $Y$ is $\dim W_f$.
\end{lem}

\begin{proof}
This is an immediate consequence of unwinding the definition of $\varphi$ 
and using \cite[Proposition 4.4]{bc18}.
\end{proof}

\begin{example}
Let $p=5$ and consider the basic Artin-Schreier curve $y^5 - y = f = x^{11}$.  Note that $n_1=5$.  To illustrate Definitions~\ref{def:omega} and \ref{def:gamma}, let us compute $\gamma_f( (0,x^j dx, 0, 0,0))$ for $j \in \{0,1,2,3,5\}$.

We immediately see that $\omega_2= \omega_3 = \omega_4 =0$, and that $\omega_1 = x^j ydx$.  The interesting calculation is computing $\omega_0$.  We have that
\[
\omega_0 = 0 + s\left( \Cartier \left( - \binom{1}{0} \omega_1 (-f) \right) \right) = s \left( \Cartier \left( x^{j+11} dx \right) \right).
\]
Now $\Cartier(x^{11+j} dx) =0$ unless $11+j \equiv 4 \pmod{5}$.  Thus $\gamma_f((0,x^j dx, 0, 0,0))$ is regular if $j = 0,1,2,5$, and $ (0,x^j dx, 0, 0,0) \in W_f$.  When $j=3$, we have that
\[
 \gamma_f(x^3 y dx)= x^3 y dx + x^{14} dx.
\]
This is not regular because of the $x^{14}dx$ term (note that $n_0 = 7$).  Hence $(0,x^3 dx, 0, 0,0) \not \in W_f$.
\end{example}

An equivalent way to describe $W_f$, which is more in line with the arguments of \cite{bc18}, is to describe it as the kernel of a linear transformation $\psi_f$.

\begin{defn}
Fix $0 \leq i \leq p-1$.
Define $m_i := (p-1-i) d -2$.  Let $r_i$ be the number of integers $s$ satisfying $s \equiv -1 \pmod{p}$ and $n_i < s \leq m_i$.  
Write these integers as 
$s_i, s_i +p, \ldots , s_i + (r_i -1)p$. 
\end{defn}

Suppose $\displaystyle \omega = \sum_i \omega_i y^i =  \gamma_f(v)$ for some $v \in V_f$.  Then writing $\displaystyle \omega_i = \sum_{j} c_{i,j} x^j dx$ with $c_{i,j} \in k$, it is straightforward to check that $c_{i,j} =0$ if $j > m_i$.  
This is a special case of \cite[Corollary~5.5]{bc18}.
Furthermore, in light of \eqref{eq:reconstruct} and Remark~\ref{rmk:othercomposition}, 
it follows that $c_{i,j} =0$ if $j \not \equiv -1 \pmod{p}$ and $n_i < j \leq m_i$.

\begin{defn} \label{def:psi}
Fix $0 \leq i \leq p-1$.    Define $\psi_{f,i} : V_f \to k^{r_i}$ by $$\psi_{f,i}(v) = (c_{i,s_i}, c_{i,s_i + p} ,\ldots c_{i,s_i + p (r_i-1)}).$$

Furthermore, set $U_f := \bigoplus_{i=0}^{p-1} k^{r_i}$ and define a map $\psi_f: V_f \to U_f$ by $$\psi_f(v) = (\psi_{f,0}(v),\psi_{f,1}(v), \ldots, \psi_{f,p-1}(v)).$$
\end{defn}

The idea is that $\psi_{f,i}$ records the coefficients of $\omega_i$ which might prevent it from being regular.  Again, note that $\psi_f$ and $U_f$ depend on the Artin-Schreier curve $y^p - y = f$; we may drop the subscripts when the curve is clear from context.

\begin{prop} \label{prop:summary}
Consider the basic Artin-Schreier curve $Y$ given by $y^p - y =f$.  Then
\[
a_Y = \dim_k W_f = \dim_k \ker(\psi_f).
\]
\end{prop}

\begin{proof}
Unwinding definitions, notably \cite[Definitions 5.6, 5.8, 6.1]{bc18}, shows the following:
we may identify $U_f$ with the vector space $\bigoplus_{i=0}^{p-1} H^0(\PP^1,M_i)$ of \cite{bc18}; and the map $\psi_f$ is identified with the map $\widetilde{g}$,
under the identification of $V_f$ with $H^0(\PP^1,\Gscr_p)$.
By \cite[Lemma 6.3]{bc18}, $W_f = \ker(\psi_f)$, and hence the $a$-number of $Y$ is $\dim_k \ker(\psi_f)$.
\end{proof}

\begin{remark}
A more general class of covers is considered in \cite{bc18}, leading to much more complicated notation.  Things simplify for basic Artin-Schreier curves.  In the notation of \cite{bc18}:

\begin{itemize}

\item  As a basic Artin-Schreier cover is ramified only above infinity, we take $S = \{\infty\}$.

\item  As we are working over $\PP^1$, there is no need for the auxiliary divisors $D_i$ introduced in \cite[\S5]{bc18}.  We may instead work with an explicit map $s : \Im \Cartier_{\PP^1} \to F_* \Omega^1_{\PP^1}$ (see Definition~\ref{def:s}) that does not introduce poles.

\item  There is no need for the auxiliary point $Q'$ when working over $\PP^1$; see \cite[Remark 3.3]{bc18}, and thus $\pi_* \ker \Cartier_Y = \Gscr_{-1} = \Gscr_0$.
\end{itemize}
\end{remark}

%Characteristic three

\section{Basic Artin-Schreier Covers in Characteristic Three} \label{sec:char3}

In this section, we will prove the $p=3$ case of Theorem~\ref{thm:basiccovers}.  We can construct the necessary covers over $k = \FF_3$.  For a positive integer $d$ that is not a multiple of $3$, Corollary~\ref{cor:maximized} gives that
\begin{equation} \label{eq:lowerbound3}
L(\{d\})= \lceil 2d/3 \rceil + \lceil d/3 \rceil - \lceil d/9 \rceil -  \lceil 4d/9\rceil.
\end{equation}
(It is elementary to rewrite the floor functions as the ceiling functions which will naturally occur in our argument.)  We will show:

\begin{prop} \label{prop:p3}
For any positive integer $d$ relatively prime to $3$, there exists $f_d \in \FF_3[x]$ of degree $d$ such that the basic Artin-Schreier curve given by $y^3 -y = f_d$ has $a$-number $L(\{d\})$.
\end{prop}

The choice of $f_d$ will depend on whether $d \equiv 1 \pmod{3}$ or $d \equiv 2 \pmod{3}$, and we will use Proposition~\ref{prop:summary} to compute the $a$-number.   

 Recall Definition~\ref{def:ni} and Lemma~\ref{lem:V}; we have $n_0 = \left \lceil \frac{2 d}{3} \right \rceil -2$, $n_1 = \left \lceil \frac{d}{3} \right \rceil -2$, and $n_2 = -2$  so that for any $f \in \FF_3[x]$ of degree $d$ 
\[
V_f = \left\{ (h_0 dx, h_1 dx,0) : \Cartier(h_i dx) =0, \text{ and }  \deg(h_i) \leq n_i \text{ for } i=0,1 \right \}.
\]
Furthermore, unwinding Definition~\ref{def:gamma}, there exist $\omega_0$ and $\omega_1$ such that
\begin{equation} \label{eq:reconstruct3}
\gamma_f ((h_0 dx, h_1 dx, 0)) = ( h_0dx + s \Cartier( h_1 f dx )) + h_1 dx y = \omega_0 + \omega_1 y.
\end{equation}
Whether this differential is regular is controlled completely by $s \Cartier( h_1 f dx )$.  In light of \eqref{eq:othercomposition} and Fact~\ref{fact:regularity}, it is regular if and only if $h_1 f dx$ has no terms of the form $x^i dx$ with $i \equiv -1 \pmod{3}$ and $i> n_0$.  The map $\psi_f$ from Definition~\ref{def:psi} records these coefficients.

\subsection{The Case \texorpdfstring{$d \equiv 1 \pmod{3}$}{d= 1 mod 3}}  Fix a positive integer $d$ with $d \equiv 1 \pmod{3}$, and let $L:= L(\{d\})$. 
We begin by recording a useful fact about ceiling functions.

\begin{lem} \label{lem:ceilinglemma}
We have that 
\[
\left \lceil \frac{2d}{9} \right \rceil+  \left \lceil \frac{ \lceil d/3 \rceil -2 }{3} \right \rceil + \left \lceil \frac{ \lceil d/3 \rceil -1 }{3} \right \rceil   = \left \lceil \frac{4d}{9} \right \rceil.
\]
\end{lem}

\begin{proof}
Replacing $d$ by $d+9$ adds $4$ to both sides, so it suffices to check this for $d \in \{1,4,7\}$.
\end{proof}

\begin{defn}
Let $b_2 := 3 \left \lceil \frac{ \lceil d/3 \rceil -1 }{3} \right \rceil -1$ be the largest integer congruent to $2$ modulo $3$ that is less than or equal to $\lceil d/3 \rceil$.
Define  $f_d := t^{d} + t^{d - b_2}$.
\end{defn}

\begin{prop} 
The basic Artin-Schreier curve $y^3- y = f_d$ has $a$-number $L(\{d\})$.  
\end{prop}

\begin{proof}
For $i \leq n_1 = \left \lceil \frac{d}{3} \right \rceil -2 $ with $i \not \equiv -1 \pmod{3}$, we consider
\[
\gamma( (0,x^i dx,0)) = s (\Cartier( x^{i + d}dx + x^{i + d -b_2} dx)) + x^i y dx.
\]
If $i \equiv 1 \pmod{3}$, as $i + d \equiv 2 \pmod{3}$ and $i + d-b_2 \not \equiv 2 \pmod{3}$, we see that
\begin{equation}
\gamma( (0,x^i dx,0)) = x^{i +d} dx + x^i y dx.
\end{equation}
If $i \equiv 0 \pmod{3}$, then $i + d \not \equiv 2 \pmod{3}$ and $i + d-b_2 \equiv 2 \pmod{3}$ and we see that
\begin{equation}
\gamma( (0,x^i dx,0)) = x^{i + d-b_2} dx + x^i y dx.
\end{equation}
These are not regular when $i+ d > n_0 $ and when $i + d-b_2 > n_0$ respectively.  But it is elementary to see that $d> n_0$ and $d-b_2 > n_0$, respectively, so these differentials are never regular.  Furthermore, observe that the terms of the form $x^j dx$ which cause $\gamma((0,x^i dx,0))$ not to be regular are all different; every integer of the form $i + d- b_2 $ with $0 \leq i \leq n_1 = \left \lceil \frac{d}{3} \right \rceil -2$ is less than $d$, while every integer of the form $i +d$ with $0 \leq i \leq n_1$ is greater than or equal to $d$.  
 
Re-interpreting this in terms of the map $\psi$, we see that
\[
T:=\left \{ \psi( (0,x^i dx,0)) : 0 \leq i \leq n_1 = \left \lceil \frac{d}{3} \right \rceil -2, \, i \not \equiv 2 \pmod{3} \right \}
\]
is a set of linearly independent elements in $\psi(V_f) \subset U_f$; the entries of vectors in $U_f$ correspond to coefficients of monomials in $\omega_0$, so each $\psi( (0,x^i dx,0))$ has exactly one non-zero entry, and these occur in different spots.   We see that
\[
\#T=\left \lceil \frac{ \lceil d/3 \rceil -1}{3} \right \rceil + \left \lceil \frac{\lceil d/3 \rceil -2 }{3} \right \rceil.
\]
Then we have that
\[
\dim_k \ker(\psi) \leq \dim_k V - \#T.
\]
Using the dimension of $V$ from Remark~\ref{rmk:dimV}, we conclude that
\begin{align*}
\dim_k \ker(\psi) &\leq  \lceil 2d/3 \rceil - \lceil 2 d/9 \rceil + \lceil d/3 \rceil - \lceil d/9 \rceil - \left( \left \lceil \frac{ \lceil d/3 \rceil -2 }{3} \right \rceil + \left \lceil \frac{ \lceil d/3 \rceil -1 }{3} \right \rceil \right)\\
&= \lceil 2d/3 \rceil + \lceil d/3 \rceil - \lceil 4d/9 \rceil - \lceil  d/9 \rceil = L
\end{align*}
using Lemma~\ref{lem:ceilinglemma}.  By Proposition~\ref{prop:summary}, $\dim_k \ker(\psi)$ is the $a$-number of the curve given by $y^p - y = f_d$.  Since $L$ is a lower bound on the $a$-number (Corollary~\ref{cor:explicitlower}), the curve has minimal $a$-number as desired.  \end{proof}

This proves Proposition~\ref{prop:p3} when $d \equiv 1 \pmod{3}$.

\subsection{The Case \texorpdfstring{$d \equiv 2 \pmod{3}$}{d = 2 mod 3}}

Fix a positive integer $d$ with $d \equiv 2 \pmod{3}$.

\begin{defn}
Let  $b_0 = 3 \left \lceil \frac{\lceil d/3 \rceil -2}{3} \right \rceil$ be the largest multiple of $3$ less than or equal to $\lceil d/3 \rceil$.  
Define $f_d = x^d + x^{d-b_0 -1}$.
\end{defn}

\begin{prop}
The basic Artin-Schreier curve $y^3 - y = f_d$ has $a$-number $L(\{d\})$.
\end{prop}

\begin{proof}  The argument is essentially the same as in the $d \equiv 1 \pmod{3}$ case.  
\end{proof}

This completes the proof of Proposition~\ref{prop:p3}. \qed

%section Oxford Group

\section{Some Basic Artin-Schreier Covers in Characteristic Five}\label{sec:oxford}
    In this section, we will work towards the $p=5$ case of Theorem~\ref{thm:basiccovers}.  We will give explicit families depending on $d$ modulo $p^2$, which were found by exhaustive searches for binomials and trinomials $f$ such that the basic Artin-Schreier curve $y^p - y = f$  has minimal $a$-number. The $a$-numbers of each family can be computed explicitly using Lemma~\ref{lem:calc-anumber}.

   \subsection{A Binomial Family} We first focus on a family of binomial polynomials $f$.  
    
	\begin{prop} \label{prop:binomial}
	    Fix a positive integer $d$ relatively prime to $5$, and write $d= 25m+\delta$ with $0\leq \delta < 25$.  Suppose $\delta \not \in \{3,7,9,16,18,22\}$, and   
let $\delta'$ be as in Table~\ref{table:deltadelta'}.  Let $f=x^{25m+\delta} + x^{15m+\delta'}$.
Then the basic Artin-Schreier curve $y^5 - y =f$ has $a$-number $L(\{d\})$. 
	\end{prop}
	
	\begin{table}[h!]
		\centering
		\begin{tabular}{l l | l l | l l | l l | l l}
			$\delta$ & $\delta'$ & $\delta$ & $\delta'$ & $\delta$ & $\delta'$ & $\delta$ & $\delta'$ & $\delta$ & $\delta'$\\
			\hline
			1 &  3 & 6 &  3 & 11 &  8 & 16 & -- & 21 & 13 \\
			2 &  1 & 7 & -- & 12 &  6 & 17 & 11 & 22 & -- \\
			3 & -- & 8 &  4 & 13 & 9 & 18 & -- & 23 & 14 \\
			4 &  2 & 9 & -- & 14 &  7 & 19 & 12 & 24 & 12
		\end{tabular}
		\captionsetup{width=0.85\textwidth}
		\caption{Values of $\delta$ and $\delta'$ giving basic Artin-Schreier covers with minimal $a$-number}
		 \label{table:deltadelta'}
	\end{table}

We will begin by collecting together all of our notation:  
\begin{defn} \label{def:oxfordnotation}
Fix a positive integer $d$ coprime to $5$.
\begin{enumerate}[(i)]
\item  Write $d = 25 m + \delta$ where $m$ is an integer and $\delta \in \{0, \ldots, 24\}$.

\item  For $i \in \{0,1,2,3,4\}$, we define $b_i = \ceil{ \frac{4-i}{5} \delta} -2$.

\item  Let $f := x^{25m + \delta} + x^{15 m + \delta'}$.

\item  Define $s_0 := 2\delta$, $s_1 := \delta + \delta'$, $s_2 := 2\delta'$, $c_0 := 3\delta$, $c_1 := 2\delta + \delta'$, $c_2 := \delta + 2\delta'$, and $c_3 := 3\delta'$.

\item Set $t := \ceil{\frac{b_1 + \delta - 3}{5}} - \ceil{\frac{b_0 - 3}{5}}$.

\item Finally, let $l_0 := \ceil{\frac 4 5\delta} = b_0 + 2$, $l_1 := \ceil{\frac 3 5\delta} = b_1 + 2$, $l_2 := \ceil{\frac{3}{25}\delta}$, $l_3 := \ceil{\frac{8}{25}\delta}$, and $L := l_0 + l_1 - l_2 - l_3$.

\item  For an integer $n$: let $U_5(n)$ denote the smallest integer congruent to $4$ modulo $5$ and greater than or equal to $n$; and
let $D_5(n)$ denote the largest integer congruent to $4$ modulo $5$ less than or equal to $n$.
\end{enumerate}
\end{defn}

\begin{lem} \label{lem:randominequalities}
Let $(\delta,\delta')$ be one of the pairs in Table~\ref{table:deltadelta'}.  Then we have that:

\begin{itemize}
\item  $D_5(\delta+b_1) < U_5(s_1)$;

\item  $D_5(s_1 + b_2 ) < U_5(s_0)$;

\item  $D_5(s_0 + b_2) < U_5(c_0)$;

\item  $D_5(\delta' + b_2) < U_5(\delta)$;

\item  $D_5(\delta + b_2) < U_5(s_1)$;

\item  $D_5(s_1 + b_3 ) < U(s_0)$.
\end{itemize}
\end{lem}

\begin{proof}
There are only finitely many pairs $(\delta,\delta')$, so we verify this case by case.
\end{proof}

 Recalling the definition of $n_i$ from Definition~\ref{def:ni}, we see that 
\[
	 n_i = 5 (4-i) m + b_i.
\]
We  easily compute: 
	\begin{align*}
		f^2 &= x^{50m + s_0} + 2x^{40m + s_1} + x^{30m + s_2}; \ {\rm and}\\
		f^3 &= x^{75m + c_0} + 3x^{65m + c_1} + 3x^{55m + c_2} + x^{45m + c_3}.
	\end{align*}
Notice that $\delta$ and $\delta'$ are automatically coprime to $5$ as $d$ is coprime to $5$, and hence none of $s_0$, $s_1$, and $s_2$ are multiples of $5$.

From Corollary~\ref{cor:maximized}, we see that
\begin{equation}
L(\text{\{}d\text{\}}) = (15m + l_1) + (20m + l_0) - (3m + l_2) - (8m + l_3) = 24m + L.
\end{equation}

For $a \in \FF_p[x]$, let $\overline{a}$ denote the sum of all monomial terms of $a$ with degree a multiple of $p$.  Notice that for any $a, b \in \FF_p[x]$, we have $\sC{a\cdot\sC{b dx}} = \overline a \cdot \sC{b dx}$.

\begin{proof}[Proof of Proposition~\ref{prop:binomial}]
Let $Y$ be the basic Artin-Schreier curve given by $y^p - y =f$.
We begin by computing the images under $\gamma$ of the natural basis of $V_f$ given by monomials.   For a non-negative integer $i$ a straightforward computation using Definitions~\ref{def:omega} and \ref{def:gamma} gives that:
\begin{align*}
        \gamma\left(\left(x^idx, 0,0,0,0\right)\right) &= x^i\phantom{y^2}dx;\\
        \gamma\left(\left(0,x^idx,0, 0, 0\right)\right) &= x^i y\phantom{\llap{a}^2}dx +  
        \phantom{3}\sC{fx^i dx}\phantom{y^2};\\
        \gamma\left(\left(0,0, x^idx, 0,  0\right)\right) &= x^i y^2dx + 2 \sC{fx^i dx} y\phantom{\llap{a}^2}  + \phantom{3}\sC{-x^i f^2 dx+  2 f\sC{fx^idx}} \\
        &= x^i y^2dx + 2 \sC{fx^idx} y\phantom{\llap{a}^2} - \phantom{3}\sC{x^i f^2 dx}\phantom{y};\\
        \gamma\left(\left(0,0,0,x^idx, 0\right)\right) &=  x^i y^3 dx + 3\sC{fx^i} y^2dx - 3\sC{x^i f^2}ydx + \sC{x^i f^3}dx.
	\end{align*}
Here we use that
\[
\sC{ 2 f \sC{f x^i dx}} = 2 \overline{f} \sC{f x^i dx} = 0,
\]
since by our choice of $f$ there are no monomial terms with exponent a multiple of $p=5$ (and with a similar argument for the fourth equality).  We now study the images in detail, breaking into separate cases for ease of reference.

\boldheader{Case 0}
	For all $i \in \{0, \ldots, 20m + b_0 \}$ with $i \not\equiv -1 \pmod 5$, we have that $ \gamma\left(\left(x^idx,0, 0, 0,  0\right)\right) = x^idx$ is regular.
 There are $16m+ l_0 - \ceil{\frac{l_0}{5}}$ such $i$, using the elementary observation that $l_0 - \ceil{\frac{l_0}{5}} = b_0 +1 - \ceil{\frac{b_0 - 3}{5}}$.

\boldheader{Case 1}
	For $i \in \{0, \ldots, 15m + b_1\}$, $i \not\equiv -1 \pmod 5$ (of which there are $12m + l_1 - \ceil{\frac{l_1}{5}}$) 
	we have
	\begin{align*}
		\gamma\left(\left(0, x^idx,0,0, 0\right)\right) &= x^i ydx + \sC{x^{25m + i + \delta} + x^{15m + i + \delta'}}dx\\
		&= x^i ydx + \begin{cases}
			x^{15m + i + \delta'}dx,\quad &\text{if }i \equiv 4-\delta' \pmod 5\\
            x^{25m + i + \delta\phantom{'}}dx,\quad &\text{if }i \equiv 4 - \delta\phantom{'} \pmod 5\\
			0,\quad &\text{otherwise}
		\end{cases}
	\end{align*}
	since $\delta' \not\equiv \delta \pmod5$. 
	 Some of these are regular and others not: a differential of the form $x^j dx$ on $Y$ is regular provided that $j \leq 20 m + b_0$.

	\boldheader{Case 2}
	For $i \in \{0, \ldots, 10m + b_2\}$, $i \not\equiv -1 \pmod 5$, we have
	\begin{align*}
		\gamma\left(\left(0,  0,x^idx, 0, 0\right)\right) &= x^i y^2dx + 2\sC{x^{25m + i + \delta} + x^{15m + i + \delta'}}ydx\\
		&\qquad + \sC{-x^{50m + i + s_0} - 2x^{40m + i + s_1} - x^{30m + i + s_2}}dx.
	\end{align*}
	Hence, we have the cases:
	\begin{alignat*}{2}
        i & \equiv 4 - 2\delta & \pmod 5 : \qquad \gamma\left(\left(0, 0,x^idx,  0, 0\right)\right) &= x^i y^2\D x - \phantom{2}x^{50m + i + s_0}dx ; \\
		i & \equiv 4 - 4\delta & \pmod 5 : \qquad \gamma\left(\left(0,0, x^idx,  0, 0\right)\right) &= x^i y^2\D x - 2x^{40m + i + s_1}dx ;\\
        i & \equiv 4 - \phantom{1}\delta & \pmod 5 : \qquad \gamma\left(\left(0,0, x^idx,  0, 0\right)\right) &= x^i y^2\D x + 2x^{25m + i + \delta}ydx + x^{30m + i + s_2}dx ;\\
        i & \equiv 4 - 3\delta & \pmod 5 : \qquad \gamma\left(\left(0,0, x^idx,  0, 0\right)\right) &= x^i y^2dx + \phantom{2}x^{15m + i + \delta'}ydx.
	\end{alignat*}
	The images are not regular for any $i \not \equiv -1 \pmod{5}$.
	
\boldheader{Case 3}	For $i \in \{0, \ldots, 5m + b_3 \}$, $i \not\equiv -1 \pmod 5$, we have
	\begin{align*}
		\gamma\left(\left(0, 0, 0, x^idx, 0\right)\right) &= x^i y^3 + 3\sC{x^{25m + i + \delta} + x^{15m + i + \delta'}}y^2dx\\
		&\qquad -3\sC{x^{50m + i + s_0} + 2x^{40m + i + s_1} + x^{30m + i + s_2}}ydx\\
		&\qquad + \sC{x^{75m + i + c_0} + 3x^{65m + i + c_1} + 3x^{55m + i + c_2} + x^{45m + i + c_3}}dx.
	\end{align*}
	So we have the cases:
	\begin{alignat*}{2}
		i & \equiv 4 - 2\delta & \pmod 5 : \qquad \gamma\left(\left(0, 0, 0, x^i dx, 0\right)\right) &= x^i y^3dx - 3x^{50m + i + s_0}y dx + 3x^{55m + i + c_2}dx;\\
        i & \equiv 4 - 4\delta & \pmod 5 : \qquad \gamma\left(\left(0, 0, 0, x^i dx, 0\right)\right) &= x^i y^3dx - \phantom{3}x^{40m + i + s_1}ydx + x^{45m + i + c_3}dx;\\
        i & \equiv 4 - \phantom{1}\delta & \pmod 5 : \qquad \gamma\left(\left(0, 0, 0, x^i dx, 0\right)\right) &= x^i y^3\D x + 3x^{25m + i + \delta}y^2 dx - 3x^{30m + i + s_2}y dx; \\
		i & \equiv 4 - 3\delta & \pmod 5 : \qquad \gamma\left(\left(0, 0, 0, x^i dx, 0\right)\right) &= x^i y^3dx + 3x^{15m + i + \delta'}y^2 dx + x^{75m + i + c_0} dx.
	\end{alignat*}
	Again, none of these images are regular.

Having now computed the images of a basis of $V_f$ under $\gamma$ (using Fact~\ref{fact:regularity}),
 we bound the dimension of the image modulo differentials regular on $Y$ by exhibiting linearly independent $\gamma(\nu)$.  Going down the table, each of the images has a ``new term'' which will make it clear that it does not lie in the span of the previous images.

\begin{table}[ht]
\centering
\begin{tabular}{ |c|c|c|  }
\hline
Case & Congruence &  New Terms	\\
\hline
$1$ & - & $ x ^{20 m + 4}dx$ to $x^{40m + D_5(\delta + b_1)}dx$ \\
\hline
$2$ & $i \equiv 4- 4 \delta \pmod{5}$ & $x^{40m + U_5(s_1)}dx$ to $x^{50m + D_5(s_1 + b_2)}dx$ \\
\hline
$2$ & $i \equiv 4- 2 \delta \pmod{5}$ & $x^{50m + U_5(s_0)}dx$ to $x^{60m + D_5(s_0 + b_2)}dx$ \\
\hline
$3$ & $i \equiv 4- 3 \delta \pmod{5}$ & $x^{75m + U_5(c_0)}dx$ to $x^{80m + D_5(c_0 +b_3)}dx$ \\
\hline
$2$ & $i \equiv 4- 3 \delta \pmod{5}$ & $x^{15m + U_5(\delta')} y dx$ to $x^{25m + D_5(\delta'+b_2)} y dx$ \\
\hline
$2$ & $i \equiv 4- \delta \pmod{5}$ & $x^{25m + U_5(\delta)} y dx$ to $x^{35m + D_5(\delta+b_2)} y dx$ \\
\hline
$3$ & $i \equiv 4- 4 \delta \pmod{5}$ & $x^{40m + U_5(s_1)} y dx$ to $x^{45m + D_5(s_1+b_3)} y dx$ \\
\hline
$3$ & $i \equiv 4- 2 \delta \pmod{5}$ & $x^{50m + U_5(s_0)} y dx$ to $x^{55m + D_5(s_0+b_3)} y dx$ \\
\hline
$3$ & $i \equiv 4- \delta \pmod{5}$ & $x^{25m + U_5(\delta)} y^2 dx$ to $x^{30m + D_5(\delta+b_3)} y^2 dx$ \\
 \hline 
\end{tabular}	
\caption{Differentials in the Image of $\gamma$}
\label{table:images}
\end{table}

For example, in Case 1 the $x^i ydx$-term is automatically regular, while a term of the form $x^j dx$ is regular provided that $j \leq n_0 = 20m + b_0$.  Now the exponents of
$x^{15m + i + \delta'}dx$ and $x^{25m + i + \delta} dx$ where $i \equiv 4 - \delta' \pmod{5}$ or $i \equiv 4-\delta \pmod{5}$ and $0 \leq i \leq   15m+b_1$ run
through every integer $k$ that is congruent to $-1$ modulo $5$ and that is between $15 m+ \delta'$ and
	 $$25m + \delta + 15m + b_1 = 40m + \delta + b_1.$$ 
In particular, there is a  $\nu = (0,x^idx,0,0,0)$ such that $\gamma(\nu) = x^j dx + ( \ldots ) y dx$ for every such $j$ ranging from $20m+4$ to $40m + D_5(\delta + b_1)$.  	A simple calculation shows that there are $4m+t$ such values $k$. These $4m+t$ images are clearly linearly independent in the quotient of the space of differentials on $Y$ by the space of differentials regular on $Y$.

In the second row of the table, the images are still of the form $x^j dx + (\ldots) ydx$, but now $j$ is larger.  Lemma~\ref{lem:randominequalities} shows that $40m + D_5(\delta+b_1)$ is less than $40m + U_5(s_1)$, so the images from the second row do not lie in the span of the images from the first row because there are new exponents.  This pattern continues in each new row: there are new terms in the image which make it clear that the new images do not lie in the spans of the previous elements, and Lemma~\ref{lem:randominequalities} guarantees that the exponents listed in Table~\ref{table:images} are actually distinct.  Thus all of the images listed in Table~\ref{table:images} are linearly independent modulo regular differentials.
	
Note that Table~\ref{table:images} contains all of the images from Case 2 and 3, none from Case 0, and some from Case 1.  Thus the dimension of $\gamma(V_f)$ modulo regular differentials is at least the number of differentials in Case 2 and 3 plus $4m +t$.  By the rank-nullity theorem, we see that
\[
\dim W_f \leq 16m+ l_0 - \ceil{\frac{l_0}{5}} + 12m + l_1 - \ceil{\frac{l_1}{5}} - (4m + t),
\]
since there are $16m + r_0$ elements of $V_f$ of the form $(x^i dx,0,0,0,0)$ and $12m +r_1$ elements of $V_f$ of the form $(0,x^idx,0,0,0)$.  We can simplify
\begin{align*}
	\dim W_f & \leq 	16m+ l_0 - \ceil{\frac{l_0}{5}} + 12m + l_1 - \ceil{\frac{l_1}{5}} - (4m + t) \\
	& = 24m + l_0 - \ceil{\frac{l_0}{5}} + l_1 - \ceil{\frac{l_1}{5}} - \ceil{\frac{b_1 + \delta - 3}{5}} + \ceil{\frac{b_0 - 3}{5}}\\
		&= 24m + l_0 + l_1 - \ceil{\frac{l_1}{5}} - \ceil{\frac{l_1 + \delta}{5}}\\
		&= 24m + l_0 + l_1 - \ceil{\frac{l_1}{5}} - \ceil{\frac{l_1 + \delta}{5}}.
	\end{align*}
	But we have
	\begin{align*}
		\ceil{\frac{l_1}{5}} &= \ceil{\frac{\ceil{\frac{3}{5}\delta}}{5}} = \ceil{\frac{3}{25}\delta} = l_2,
		\intertext{and}
		\ceil{\frac{l_1 + \delta}{5}} &= \ceil{\frac{\ceil{\frac{3}{5}\delta + \delta}}{5}} = \ceil{\frac{\frac 3 5 \delta + \delta}{5}} = \ceil{\frac{8}{25}\delta} = l_3,
	\end{align*}
	so
	\begin{align*}
		\dim W_f &\leq 24m + l_0 + l_1 - l_2 - l_3\\
		&= 24m + L = L(\{d\}).
	\end{align*}
By Lemma~\ref{lem:calc-anumber}, $\dim W_f = a_Y$, so using Fact~\ref{fact:lowerbound} we conclude that $a_Y = L(\{d\})$ as desired.
\end{proof}

\begin{remark}
For the excluded values of $\delta$, we cannot find a basic Artin-Schreier curve $y^p - y =f$ with $f$ a binomial of degree $d = 25m+\delta$ having minimal $a$-number. 
\end{remark}
	
\subsection{Other Families}
To complete the proof of Theorem~\ref{thm:basiccovers} when $p=5$, we must consider curves $y^5 - y =f$ for more general $f$.  
In particular, we can use trinomials $f = x^{25m + \delta} + x^{15m + \delta'} + x^{5m + \delta''}$ for $\delta', \delta'' \in \integers$.
Examples of $\delta'$ and $\delta''$ which give a basic Artin-Schreier curve with minimal $a$-number are listed in Table~\ref{table:trinomial} for the values of $d = 25m + \delta$ where binomials do not suffice.

	\begin{table}[!ht]
		\centering
		\begin{tabular}{l l l | l l l}
			$\delta$ & $\delta'$ & $\delta''$ & $\delta$ & $\delta'$ & $\delta''$\\
			\hline
			3 & 4 & 1 & 16 & 14 & -1\\
			7 & 6 & 3 & 18 & 14 &  6\\
			9 & 7 & 3 & 22 &  1 &  3\\
		\end{tabular}
		\captionsetup{width=0.85\textwidth}
		\caption{Values of $\delta'$ and $\delta''$ giving minimal $a$-number} \label{table:trinomial}
	\end{table}

  The analysis follows the same approach as the proof of Proposition~\ref{prop:binomial}. Although it is more complex due to the larger number of terms involved, the image of $\gamma$ still has a very similar form -- with all terms constant in $y$ having regular images, together with all but a certain family of the terms linear in $y$. All of the other basis terms again have linear independent images modulo differentials regular on $Y$.
  As an analysis for a (different) family of trinomials is given in full detail in \S\ref{sec:boston}, we do not give the details here. 

%Section Boston Group

\section{Trinomial Artin-Schreier Covers in Characteristic Five} \label{sec:boston}

In this section, we complete the proof of the $p=5$ case of Theorem~\ref{thm:basiccovers}.
We do so by studying four families of basic Artin-Schreier covers, depending on the residue class of $d$ modulo $5$.  These realize the lower bound for sufficiently large $d$, and are listed in Table~\ref{table:families}. 

\begin{table}[ht]
\centering
\begin{tabular}{ |c||c|  }
\hline
$d$ & $f$ for family \\\hline\hline
$5n+1$ & $x^{5n+1}+x^{5n-1}+x^{5n-5\floor{\frac{(n+2)2}{5}}+4}$  \\ \hline 
$5n+2$ & $x^{5n+2}+x^{5n+1}+x^{5n-5\floor{\frac{(n-1)2}{5}}-1}$ \\\hline
$5n+3$ & $x^{5n+3}+x^{5n+2}+x^{5n-5\floor{\frac{(n-1)2}{5}}-1}$ \\ \hline
$5n+4$ & $x^{5n+4}+x^{5n+2}+x^{5n-5\floor{\frac{(n+1)2}{5}}+3}$ \\ \hline
\end{tabular}
\caption{Polynomials for families whose $a$-numbers realize the lower bound}
\label{table:families}
\end{table}

\begin{remark}
Table~\ref{table:families} assumes that $n\geq 1$.  For the remaining cases, use $x^2$, $x^3 + x^2$, and $x^4$ for $f$. 
\end{remark}

\begin{prop} \label{prop:p5trinomial}
Fix a positive integer $d$ co-prime to $5$, and let $Y$ be the basic Artin-Schreier curve given by $y^5 - y = f$ where $f$ is listed in Table~\ref{table:families}.  Then $Y$ has $a$-number $L(\{d\})$.
\end{prop}

The proofs for each family are very similar, so we only show the proof for the $d=5n+1$ case.  In this case, we assume $ n \geq 3$ and $Y$ is given by
$$y^5 - y = f = x^{5n+1}+x^{5n-1}+x^{5n-5\floor{\frac{(n+2)2}{5}}+4}.$$

When $p=5$ and $d = 5n+1$, Corollary~\ref{cor:explicitlower} gives 
$$L(\{d\})=9n-\left( \floor{\frac{2n}{5}}+\floor{\frac{7n+1}{5}}+\floor{\frac{12n+2}{5}} \right).$$

Recalling Definition~\ref{def:ni}, we see that $n_0=4n-1$, $n_1=3n-1$, $n_2=2n-1$, $n_3=n-1$ and $n_4=-2$. 
Thus the vector space $V$ from Lemma~\ref{lem:V} is
\[
V  = \left\{ ( \nu_0, \nu_1, \nu_2, \nu_{3},0 ) : \nu_i = h_i dx, \, \Cartier(\nu_i)=0, \text{ and } \deg(h_i) \leq (4-i) n -1  \text{ for } 0 \leq i \leq 3 \right\}.
\]

We will compute the $a$-number of $Y$ using the linear transformation $\psi : V \to U$ from Definition~\ref{def:psi}.  The $a$-number of $Y$ is the dimension of the kernel of $\psi$ by Proposition~\ref{prop:summary}, and the rank nullity theorem says that 
$$\dim(\ker(\psi))+\dim(\Ima(\psi)) = \dim(V).$$
We will give a lower bound on $\dim(\Ima(\psi))$, thus obtaining an upper bound for the $a$-number which will equal the lower bound $L(\{d\})$.  Hence the $a$-number of $Y$ will be $L(\{d\})$ as desired.  

\subsection{A Proof}
We aim to produce a large number of elements of $V$ whose images under $\psi$ are linearly independent.  Equivalently, we wish to produce differentials  in the image of $\gamma$ that are linearly independent modulo regular differentials.  It will be helpful to introduce the following notation for elements of $V$.

\begin{defn} \label{def:nuab}
For integers $i$ and $j$ with $ 0 \leq i \leq 3$ and $0\leq j \leq (4-i) n -1$ with $j \not \equiv -1 \pmod{5}$, let 
$\nu_{i,j}$ be the element of $V$ which has $x^j dx$ in the $i$th component and $0$ elsewhere.
\end{defn}

It is immediate that $\nu_{i,j}$ form a basis for $V$; see Remark~\ref{rmk:dimV}.
In order to analyze linear independence in the image of $\gamma$, we also introduce an ordering on monomial differentials.  

\begin{defn} \label{def:ordering}
For integers $a_1, a_2, b_1, b_2$ with $0 \leq b_1, b_2 \leq 4$, say that $x^{a_1} dx y^{b_1}  < x^{a_2}dx y^{b_2} $ if $b_1 < b_2$ or if $b_1 = b_2$ and $a_1 > a_2$. 
\end{defn}

Recalling Definition~\ref{def:omega}, we see that:
\begin{align}
\gamma(\nu_{0,j}) &= x^j dx  \label{eq:nu0}; \\
\gamma(\nu_{1,j}) &= s\left(\mathscr{C}\left(x^{5n+1+j}dx+x^{5n-1+j}dx+x^{5n-5\floor*{\frac{2(n+2)}{5}}+4+j}dx\right)\right) + x^j dx y \label{eq:nu1};\\
\gamma(\nu_{2,j}) &= s ( \Cartier (\omega_{j,0}))  + s\left( \Cartier \left( \omega_{j,1} \right)\right) y + x^j dx y^2; \label{eq:nu2} 
\end{align}
where we have 
\[
\omega_{j,1} := 2 x^{5n+1+j}dx+x^{5n-1+j}dx+x^{5n-5\floor*{\frac{2(n+2)}{5}}+4+j}dx,
\]
and we have 
\begin{align*}
\omega_{j,0} :=  s \left( \Cartier \left(  \omega_{j,1} \right) \right) \left(x^{5n+1+j}dx+x^{5n-1+j}dx+x^{5n-5\floor*{\frac{2(n+2)}{5}}+4+j}dx \right) -  x^{10n+2+j}dx-2x^{10n+j}dx \\-2x^{10n-5\floor{\frac{2(n+2)}{5}}+5+j}dx-x^{10n-2+j}dx -2x^{10n-5\floor{\frac{2(n+2)}{5}}+3+j}dx
- x^{10n-10\floor{\frac{2(n+2)}{5}}+8+j}dx.
\end{align*} 

Recalling Remark~\ref{rmk:othercomposition}, we see that $\gamma(\nu_{0,j})$ is always regular, and that $\gamma(\nu_{1,j})$ and $\gamma(\nu_{2,j})$ are not regular if they have a term (not involving $y$) of the form $x^i dx$ with $i > 4n-1$ and $i \equiv 4 \pmod{5}$, or of the form $x^i dx y$ with $i> 3n-1$ and $i \equiv 4 \pmod{5}$. 

To simplify notation, we set
\begin{equation}
\alpha := 5n-5\left \lfloor {\frac{2(n+2)}{5}} \right \rfloor+4.
\end{equation}
We often break into cases based on $n$ modulo $5$ to obtain clean answers.  We write $n = 5k+\delta$ with $\delta \in \{0,1,2,3,4\}$.  For reference, Table~\ref{table:alpha} expresses $\alpha$ in terms of $k$ in each of these cases.

\begin{table}[ht]
\begin{tabular}{| c | c | c| c| c| c|}
\hline
 & $n = 5k$ & $n = 5k+1$ & $n = 5k+2$ & $n=5k+3$ & $n = 5k+4$\\
\hline
$\alpha$ & $15k+4$ & $15 k +4 $ & $15k +9$ & $15k+9$ & $15k+14$\\
\hline
\end{tabular}
\caption{$\alpha$ in terms of $k$} \label{table:alpha}
\end{table}

The next step is to write down $\gamma(\nu)$ for some well-chosen types of $\nu \in V$; we will do this type by type, and record the results in Table~\ref{table:imagesummary}.  

The images are differentials which are not regular, and the images (with appropriate restrictions on $j$) will be linearly independent in the quotient of the space of differentials on $Y$ by the space of regular differentials on $Y$.  The column $\gamma(\nu)$ shows the ``largest relevant terms'': the omitted terms are either terms which are regular or which are smaller in the ordering of Definition~\ref{def:ordering}. 
Table~\ref{table:counting} collects the number of images of each type.   

\begin{table}[ht]
\begin{tabular} {| c | c | c | c | c |}
\hline 
Type & $\nu$ & Congruence & Potential Range & $\gamma(\nu)$  \\
\hline
A & $\nu_{2,j}$ & $ j \equiv 2 \pmod{5}$ & $0 \leq j \leq 2n-1$ & $- x^{10n + 2 + j}dx + \ldots $ \\
\hline
B & $\nu_{2,j} - \eta_{2,j}$ & $j \equiv 1 \pmod{5}$ & $0 \leq j \leq 2n-1$ & $- 2 x^{5n-1 + \alpha +j} dx+ \ldots$ \\
\hline
C & $\nu_{1,j}$ & $j \equiv 3 \pmod{5}$ &  $0 \leq j \leq 3n-1$ & $x^{5n+1+j} dx + \ldots$ \\
\hline
D & $\nu_{1,j}$ & $ j \equiv 0 \pmod{5}$ & $0 \leq j \leq 3n-1$ & $x^{\alpha +j }dx + \ldots $\\
\hline
E & $\nu_{3,j}$ & $ j \equiv 2 \pmod{5}$ & $0 \leq j \leq n-1$ & $2 x^{10n + 2 + j}dx  y+ \ldots $ \\
\hline
F & $\nu_{3,j} - \eta_{3,j}$ & $j \equiv 1 \pmod{5}$ & $0 \leq j \leq n-1$ & $4 x^{5n-1 + \alpha +j} dx y+ \ldots$ \\
\hline
G & $\nu_{2,j}$ & $ j \equiv 3 \pmod{5}$ & $0 \leq j \leq 2n-1$ & $2 x^{5n+1+j} dx  y + \ldots$ \\
\hline
H & $\nu_{2,j}$ & $j \equiv 0 \pmod{5}$ & $0 \leq j \leq 2n-1$ & $ 2x ^{\alpha+j} dx  y + \ldots$ \\
\hline
I & $\nu_{3,j}$ & $j \equiv 3 \pmod{5}$ & $0 \leq j \leq n-1$ & $ 3x ^{5n+1+j} dx  y^2 + \ldots$ \\
\hline
J & $\nu_{3,j}$ & $j \equiv 0 \pmod{5}$ & $0 \leq j \leq n-1$ & $ 3x ^{\alpha+j} dx  y^2 + \ldots$ \\
\hline
\end{tabular}
\caption{Elements in the Image of $\gamma$} \label{table:imagesummary}
\end{table}

\begin{table} [ht]
\begin{tabular}{| c | c | c| c| c| c|}
\hline
Case & $n = 5k$ & $n = 5k+1$ & $n = 5k+2$ & $n=5k+3$ & $n = 5k+4$\\
\hline
A & $2k$ & $2k$ & $2k+1$ & $2k+1$ & $2k+2$\\
\hline
B & $2k$ & $2k+1$ & $2k+1$ & $2k+1$ & $2k+2$\\
\hline
C & $3k$ & $3k$ & $3k+1$ & $3k+2$ & $3k+2$\\
\hline
D & $k$ & $k+1$ & $k+1$ & $k+1$ & $k+1$\\
\hline
E & $k$ & $k$ & $k$ & $k+1$ & $k+1$\\
\hline
F & $k$ & $k$ & $k+1$ & $k+1$ & $k+1$\\
\hline
G & $2k$ & $2k$ & $2k+1$ & $2k+1$ & $2k+1$\\
\hline
H & $2k$ & $2k+1$ & $2k+1$ & $2k+2$ & $2k+2$\\
\hline
I & $k$ & $k$ & $k$ & $k$ & $k+1$\\
\hline
J & $k$ & $k+1$ & $k+1$ & $k+1$ & $k+1$\\
\hline
Total & $16k$ & $16k+4$ & $16k+8$ & $16k+11$ & $16k+14$\\
\hline
\end{tabular}
\caption{Number of Images of Each Type}
\label{table:counting}
\end{table}

\boldheader{Type A}  Recall that to define $\nu_{2,j}$, we need that $0 \leq j \leq 2n-1$.  The congruence condition $j \equiv 2 \pmod{5}$ means that in \eqref{eq:nu2}, only the first term in $\omega_{j,0}$ has exponent congruent to $4$ modulo $5$.  Thus we have that
\[
\gamma(\nu_{2,j}) = - x^{10n+2+j} dx + x^j dx y^2.
\]
The first term is not regular as the exponent is too big (larger than $4n-1$), while the second is regular.  Thus the first term is the ``most significant term'' and the exponent $ 10n+2+j$ ranges from $10n+4$ to the largest integer congruent to $4$ modulo $5$ and less than or equal to $12n-1$.  Table~\ref{table:A} shows the range and number of exponents.

\begin{table}[ht]
\centering
\begin{tabular}{|c|c|c|c|} 
 \hline
 $n$ & Minimal Exponent & Maximal Exponent & Number\\
 \hline
 $5k$ & $50k+4$ & $60k-1$ & $2k$\\ \hline
 $5k+1$ & $50k+14$ & $ 60k+9$ & $2k$\\ \hline
 $5k+2$ & $50k+24$ & $ 60k+24$ & $2k+1$\\ \hline
 $5k+3$ & $50k+34$ & $ 60k+34$ & $2k+1$\\ \hline
 $5k+4$ & $50k+44$ & $ 60k+49$ & $2k+2$\\ \hline
\end{tabular}
\caption{Type A}
\label{table:A}
\end{table}

\boldheader{Type B}  From the congruence condition, we have that
\[
\gamma(\nu_{2,j}) = - x^{10n-2+j} dx - 2 x^{\alpha + 5n-1 +j} dx - x^{2 \alpha +j} dx + x^j dx y^2.
\]
The complication comes from the fact that we want $-2 x^{\alpha + 5n-1 +j} dx$ to be the largest significant term, so we cannot simply work with $\nu_{2,j}$.

\begin{lem}
For $j \equiv 1 \pmod{5}$ with $0 \leq j \leq 2n-1$, there exists $\eta_{2,j} \in V$ such that
\[
\gamma(\nu_{2,j} - \eta_{2,j}) = -2 x^{5n-1+\alpha+j} dx+ \ldots,
\]
where the omitted terms are either regular or are smaller than $x^{5n-1+\alpha+j} dx$.
\end{lem}

\begin{proof}
If $ 0 \leq 2 \alpha+j- (5n+1) \leq 3n-1$, we take $\eta_{2,j} = - \nu_{1,2 \alpha+j- (5n+1)}$.  
Note that $2 \alpha+j- (5n+1) \equiv 3 \pmod{5}$.  We compute that
\[
\gamma(\nu_{1,2 \alpha+j- (5n+1)}) =  x^{2 \alpha + j}+ x^{2 \alpha+j- (5n+1)} dx y.
\]
Thus we have
\[
\gamma(\nu_{2,j})  + 2 \gamma(\nu_{1,2 \alpha+j- (5n+1)}) = -2 x^{\alpha + 5n-1 +j} dx - x^{10n-2+j} dx +2x^{2 \alpha+j- (5n+1)} dx y +  x^j dx y^2,
\]
where the third and fourth terms are regular and the second is smaller than the first term in the ordering as desired.  

\begin{table}[ht]
\begin{tabular}{|| c | c | c| c| c| c||}
\hline
 & $n = 5k$ & $n = 5k+1$ & $n = 5k+2$ & $n=5k+3$ & $n = 5k+4$\\
\hline
$5n+1-2 \alpha$ & $-5k -7$ & $-5k-2$ & $-5k-7$ & $-5k-2$ & $-5k-7$\\ 
\hline
$8n-2\alpha$ & $10k - 8$ & $10k$ & $10k-2$ & $10k+6$ & $10k+4$  \\
\hline
$2n-1$ & $10k-1$ & $10k +1$ & $10k+3$& $10k+5$ & $10k + 7$\\
\hline
\end{tabular}
\caption{Bounds} 
\label{table:bounds}
\end{table}

There are only a few cases where $0 \leq 2 \alpha +j - (5n+1) \leq 3n-1$ fails and we must make another choice of $\eta_{2,j}$.  Table~\ref{table:bounds} shows that if $j \geq 0$ then
$5n+1-2 \alpha \leq j$.  Furthermore, if $j\equiv 1 \pmod{5}$ and $ j \leq 2n-1$ it shows that $j \leq 8n- 2 \alpha$ except in four cases.  We list the appropriate choice of $\eta_{2,j}$ in each of these exceptional cases.

\begin{itemize}
\item If $n= 5k$, $j = 10 k -4$, then $\eta_{2,j} = 2 \nu_{2,1} + 2 \nu_{1,5k+8}$.

\item If $n=5k+1$, $j = 10k+1$, then $\eta_{2,j} = 2 \nu_{2,1} + 2 \nu_{1,5k+3}$.

\item If $n= 5k+2$, $j =  10k +1$, then $\eta_{2,j} = 2 \nu_{2,1} + 2 \nu_{1,5k+8}$.

\item  If $n = 5k+4$, $j = 10k+6$, then $\eta_{2,j} = 2\nu_{2,1} +  2 \nu_{1,5k+8} $.
\end{itemize}
A direct computation then shows that $\gamma(\nu_{2,j} - \eta_{2,j})$ has the desired form.
\end{proof}

\begin{table}[ht]
\centering
\begin{tabular}{|c|c|c|c|} 
 \hline
 $n$ & Minimal Exponent & Maximal Exponent & Number of Exponents\\
 \hline 
 $5k$ & $40k+4$ & $50k-1$ & $2k$ \\  \hline
 $5k+1$ & $40k+9$ & $50k+9$ & $2k+1$\\ \hline
 $5k+2$ & $40k+19$ & $50k+19$& $2k+1$ \\ \hline
 $5k+3$ & $40k+24$ & $50k+24$ & $2k+1$\\ \hline
 $5k+4$ & $40k+34$ & $50k+39$ & $2k+2$\\ \hline
\end{tabular}
\caption{Type B}
\label{table:B}
\end{table}

\boldheader{Type C}  
The congruence condition on $j$ makes 
\[
\gamma(\nu_{1,j}) = s\left(\mathscr{C}\left(x^{5n+1+j}dx+x^{5n-1+j}dx+x^{5n-5\floor*{\frac{2(n+2)}{5}}+4+j}dx\right)\right) + x^j dx y =  x^{5n+1+j} dx + x^j dx y.
\]
The exponent $5n+1+j$ ranges from $5n+4$ to the largest integer congruent to $4$ modulo $5$ and less than or equal to $8n$.  Table~\ref{table:C} shows the range and number of exponents.

\begin{table}[ht]
\centering
\begin{tabular}{ |c|c|c|c|  }
\hline
$n$ & Minimal Exponent & Maximal Exponent & Number of Exponents\\
\hline
$5k$ & $25k+4$ & $40k-1$ & $3k$\\ \hline
$5k+1$ & $25k+9$ & $40k+4$ & $3k$\\ \hline
$5k+2$ & $25k+14$ & $40k+14$ & $3k+1$\\ \hline
$5k+3$ & $25k+19$ & $40k+24$ & $3k+2$\\ \hline
$5k+4$ & $25k+24$ & $40k+29$ & $3k+2$\\ \hline
\end{tabular}
\caption{Type C} \label{table:C}
\end{table}

\boldheader{Type D}  The congruence condition $j \equiv 0 \pmod{5}$ means that
\[
\gamma(\nu_{1,j}) = x^{5n-1+j} dx + x^{\alpha +j} dx + x^j dx y.
\]
The second term has smaller exponent than the first term, and $\alpha+j$ ranges from $\alpha$ to the smallest multiple of $5$ less than or equal to $\alpha + 3n-1$.  However, if $\alpha +j \leq 4n-1$ then $-x^{\alpha+j} dx$ is regular, while if $\alpha+j \geq 5n+4$ then $x^{\alpha+j}$ has already appeared for Type C.  Table~\ref{table:D} shows the range and number of exponents we consider to obtain new differentials which are not regular.  

\begin{table}[ht]
\centering
\begin{tabular}{|c|c|c|c|} 
 \hline
 $n$ & Minimal Exponent & Maximal Exponent & Number\\
 \hline 
 $5k$ & $20k+4$ & $25k-1$ & $k$\\ \hline
 $5k+1$ & $20k+4$ & $25k+4$ & $k+1$\\ \hline
 $5k+2$ & $20k+9$ & $25k+9$ & $k+1$\\ \hline
 $5k+3$ & $20k+14$ & $25k+14$ & $k+1$\\ \hline
 $5k+4$ & $20k+19$ & $25k+19$ & $k+1$\\ \hline
\end{tabular}
\caption{Type D}
\label{table:D}
\end{table}

\boldheader{Types E-H} 
 These types are essentially the same as Types A-D, except that the range of allowable $j$ is smaller as we are working with $\nu_{3,j}$ instead of $\nu_{2,j}$ (and $\nu_{2,j}$ instead of $\nu_{1,j}$), and that a differential $x^i dx y$ is regular provided $i \leq 3n-1$.

\begin{table}[ht]
\centering
\begin{tabular}{|c|c|c|c|} 
 \hline
 $n$ & Minimal Exponent & Maximal Exponent & Number\\
 \hline 
 $5k$ & $50k+4$ & $ 55k-1$ & $k$\\ \hline
 $5k+1$ & $50k+14$ & $ 55k+9$ & $k$\\ \hline
 $5k+2$ & $50k+24$ & $ 55k+19$ & $k$\\ \hline
 $5k+3$ & $50k+34$ & $ 55k+34$ & $k+1$\\ \hline
 $5k+4$ & $50k+44$ & $ 55k+44$ & $k+1$\\ \hline
\end{tabular}
\caption{Type E}
\label{table:E}
\end{table}

\begin{table}[ht]
\centering
\begin{tabular}{|c|c|c|c|} 
 \hline
 $n$ & Minimal Exponent & Maximal Exponent & Number\\
 \hline 
 $5k$ & $40k+4$ & $ 45k-1$ & $k$\\ \hline
 $5k+1$ & $40k+9$ & $ 45k+4$ & $k$\\ \hline
 $5k+2$ & $40k+19$ & $ 45k+19$ & $k+1$\\ \hline
 $5k+3$ & $40k+24$ & $ 45k+24$ & $k+1$\\ \hline
 $5k+4$ & $40k+34$ & $ 45k+34$ & $k+1$\\ \hline
\end{tabular}
\caption{Type F}
\label{table:F}
\end{table}

\begin{table}[ht]
\centering
\begin{tabular}{|c|c|c|c|} 
 \hline
 $n$ & Minimal Exponent & Maximal Exponent & Number\\
 \hline 
 $5k$ & $25k+4$ & $ 35k-1$ & $2k$\\ \hline
 $5k+1$ & $25k+9$ & $ 35k+4$ & $2k$\\ \hline
 $5k+2$ & $25k+14$ & $ 35k+14$ & $2k+1$\\ \hline
 $5k+3$ & $25k+19$ & $ 35k+19$ & $2k+1$\\ \hline
 $5k+4$ & $25k+24$ & $ 35k+24$ & $2k+1$\\ \hline
\end{tabular}
\caption{Type G}
\label{table:G}
\end{table}

\begin{table}[ht]
\centering
\begin{tabular}{|c|c|c|c|} 
 \hline
 $n$ & Minimal Exponent & Maximal Exponent & Number\\
 \hline 
 $5k$ & $15k+4$ & $ 25k-1$ & $2k$\\ \hline
 $5k+1$ & $15k+4$ & $ 25k+4$ & $2k+1$\\ \hline
 $5k+2$ & $15k+9$ & $ 25k+9$ & $2k+1$\\ \hline
 $5k+3$ & $15k+9$ & $ 25k+14$ & $2k+2$\\ \hline
 $5k+4$ & $15k+14$ & $ 25k+19$ & $2k+2$\\ \hline
\end{tabular}
\caption{Type H}
\label{table:H}
\end{table}

\boldheader{Types I-J}
These are the same as Type C and D, except that they use $\nu_{3,j}$ so $0 \leq j \leq n-1$.

\begin{table}[ht]
\centering
\begin{tabular}{|c|c|c|c|} 
 \hline
 $n$ & Minimal Exponent & Maximal Exponent & Number\\
 \hline 
 $5k$ & $25k+4$ & $ 30k-1$ & $k$\\ \hline
 $5k+1$ & $25k+9$ & $ 30k+4$ & $k$\\ \hline
 $5k+2$ & $25k+14$ & $ 30k+9$ & $k$\\ \hline
 $5k+3$ & $25k+19$ & $ 30k+14$ & $k$\\ \hline
 $5k+4$ & $25k+24$ & $ 30k+24$ & $k+1$\\ \hline
\end{tabular}
\caption{Type I}
\label{table:I}
\end{table}

\begin{table}[ht]
\centering
\begin{tabular}{|c|c|c|c|} 
 \hline
 $n$ & Minimal Exponent & Maximal Exponent & Number\\
 \hline 
 $5k$ & $15k+4$ & $ 20k-1$ & $k$\\ \hline
 $5k+1$ & $15k+4$ & $ 20k+4$ & $k+1$\\ \hline
 $5k+2$ & $15k+9$ & $ 20k+9$ & $k+1$\\ \hline
 $5k+3$ & $15k+9$ & $ 20k+9$ & $k+1$\\ \hline
 $5k+4$ & $15k+14$ & $ 20k+14$ & $k+1$\\ \hline
\end{tabular}
\caption{Type J}
\label{table:J}
\end{table}

\begin{lem} \label{independencelemma}
The images under $\gamma$ of the differentials of Types A-J  are linearly independent in the space of  differentials on $Y$ modulo regular differentials on $Y$.
\end{lem}

\begin{proof}
Suppose that
\[
c_1 \gamma(\nu_1) + c_2 \gamma(\nu_2) + \ldots + c_r \gamma(\nu_r)
\]
is zero modulo regular differentials, where $c_i \in \FF_5$ and the $\nu_1, \nu_2, \ldots \nu_r $ are in the $\nu$ column of Table~\ref{table:imagesummary} and are chosen with appropriate restrictions on $j$ as in the case-by-case analysis so the ``largest relevant terms'' of $\gamma(\nu_1)$ through $\gamma(\nu_r)$ are distinct.   Without loss of generality, order them so the largest relevant terms are increasing.  But then the largest relevant term of $\gamma(\nu_r)$ is not regular, but also does not occur in any of the other terms of the linear combination.  Thus $c_r=0$.  Repeating this, we see $c_1 = c_2 = \ldots = c_r=0$ and hence obtain linear independence.  
\end{proof}

\begin{lem} \label{countinglemma}
We have that
\[
\dim \Im(\psi) \geq \begin{cases}
 16k & \text{ when } n = 5k\\
 16k+4 & \text{ when } n=5k+1\\
 16k+8 & \text{ when }n=5k+2\\
 16k+11 & \text{ when }n= 5k+3\\
 16k+14 & \text{ when }n = 5k+4.
\end{cases}
\]
\end{lem}

\begin{proof}
Table~\ref{table:counting} shows how many differentials there are of each type, with the restrictions on $j$ imposed in the case-by-case analysis to make the largest relevant terms distinct.  As the images under $\gamma$ of these differentials are linearly independent modulo regular differentials by Lemma~\ref{independencelemma}, the image of $\gamma$ modulo differentials regular on $Y$ has dimension at least the total from Table~\ref{table:counting}.   But as $\psi(\nu)$ was defined to record all of coefficients of $\gamma(\nu)$ which would prevent it from being regular, we are done.
\end{proof}

\begin{proof}[Proof of the $d = 5n+1$ Case of Theorem~\ref{prop:p5trinomial}]
Using the rank-nullity theorem, the dimension of $V$ from Remark~\ref{rmk:dimV}, and Lemma~\ref{countinglemma}, we obtain an upper bound on the $a$-number of $Y$ as shown in Table~\ref{table:dimensions}.  Since this upper bound equals the lower bound $L(\{d\})$, we conclude $a_Y = L(\{d\})$.
\end{proof}

\begin{table}
\centering
\begin{tabular}{|c |c |c | c |} 
 \hline
 $n$ & $\dim(V)$ & $a$-number Upper Bound & $L(\{d\})$ \\ 
 \hline
 $5k$ & $40k$ & $24k$ & $24k$\\ \hline
 $5k+1$ & $40k+10$ & $24k+6$ & $24k+6$ \\ \hline
 $5k+2$ & $40k+18$ & $24k+10$ & $24k+10$ \\ \hline
 $5k+3$ & $40k+26$ & $24k+15$ & $24k+15$ \\ \hline
 $5k+4$ & $40k+34$ &$ 24k+20$ & $24k+20$ \\ 
 \hline
\end{tabular}
\caption{Bounds on the $a$-number, $d = 5n+1$}
\label{table:dimensions}
\end{table}

%\bibliographystyle{amsalpha}
%\bibliography{minimal}

\providecommand{\bysame}{\leavevmode\hbox to3em{\hrulefill}\thinspace}
\providecommand{\MR}{\relax\ifhmode\unskip\space\fi MR }
\providecommand{\MRhref}[2]{%
  \href{http://www.ams.org/mathscinet-getitem?mr=#1}{#2}
}
\providecommand{\href}[2]{#2}

\end{document}